\documentclass[12pt,a4paper]{article}
\usepackage{setspace}
\usepackage{amsfonts}
\usepackage{epsfig}
\usepackage{latexsym}
\usepackage{amsthm}
\usepackage{amsmath,stmaryrd}
\usepackage{amssymb}
\usepackage{array}
\usepackage{enumerate}
\usepackage{tikz}
\usepackage{tikz-cd}
\usetikzlibrary{shadings,patterns}
\usepackage{graphicx,graphics}
\newtheorem{theorem}{Theorem}[section]
\newtheorem{lemma}[theorem]{Lemma}
\newtheorem{proposition}[theorem]{Proposition}

\newtheorem{corollary}[theorem]{Corollary}
\newtheorem{definition}[theorem]{Definition}
\newtheorem{remark}[theorem]{Remark}
\newtheorem{observation}[theorem]{Observation}
\begin{document}

\title{\bf On the deep commuting graph of a finite group}
\author{Sumana Hatui \and Sanjay Mukherjee \and Kamal Lochan Patra}
\date{}
\maketitle

\begin{abstract}
Let $G$ be a finite group and let $\tilde{G}$ be a Schur cover of $G$. The deep commuting graph $\Delta_D(G)$ of $G$ is a simple graph with vertex set $G$, where two distinct vertices are adjacent if their pre-images commute in $\tilde{G}$. The deep commuting graph of a finite group was first introduced in [P. J. Cameron and B. Kuzma, Between the enhanced power graph and the commuting graph, {\it J. Graph Theory} {\bf 102} (2023), no. 2, 295--303], where the authors have shown  that $\Delta_D(G)$ is fixed irrespective of the choice of the Schur cover $\tilde{G}$. In this paper,  we first prove that $\Delta_D(G)$ is complete if and only if $G$ is cyclic. Also, we classify finite simple groups, symmetric groups and alternating groups, for which $\Delta_D(G)$ is perfect. In addition, explore several other properties of $\Delta_D(G)$ like Eulerianess, universality and connectedness of reduced deep commuting graphs. 

Next, we classify the finite abelian groups for which deep commuting graphs coincide with  enhance power graphs. We  also characterize the dominant vertices for the  deep commuting graphs of finite abelian groups and  examine the connectedness of the associated reduced deep commuting graphs. These properties of the deep commuting graphs for the non abelian groups like symmetric groups, alternating groups, dihedral groups, generalized quaternion group and Heisenberg groups are also discussed.\\



\end{abstract}

\section{Introduction}

There are various graphs associated with a group which have been studied in the literature, e.g.~Cayley graph, commuting graph, power graph, enhanced power graph etc. In $1955$, Brauer and Fowler introduced the notion of commuting graph of a group \cite{BF}, while examining the finiteness of the type of groups of even order whose centralizer follow certain conditions. For a group $G$, the \textit{commuting graph} of $G$ denoted by $\Delta(G)$, is a simple graph with vertices being the elements of $G$ and two distinct vertices are adjacent in $\Delta(G)$ if they commute in $G$. Several algebraic and combinatorial properties of the commuting graph have been studied over the years. In addition, the commuting graph has also been extremely useful in the study of commuting probability of a group \cite{eberhard, guralnick}. For a further exposer to the recent developments on commuting graphs, we refer to the survey papers \cite{arvind, cameron}.

The notion of directed power graph of a group was introduced by Kelarev and Quinn in \cite{ker1}. The underlying undirected graph, simply termed the `power graph' of that group, was first considered by Chakrabarty et al.~in \cite{ivy}. The {\it power graph} of a group $G$, denoted by $\mathcal{P}(G)$, is the simple graph with vertex set $G$, and for $u,v\in G$ with $u\neq v$, $u$ is adjacent to $v$ in $\mathcal{P}(G)$, if $v=u^{\alpha}$ or $u=v^{\beta}$ for some $\alpha,\beta\in \mathbb{N}$. For a detailed overview of several combinatorial studies on power graphs, we refer to the survey papers \cite{AKC, KSCC} and the references therein.

The idea of the enhanced power graph of a group as a super graph of the power graph, was introduced in $2017$ by Aalipour et al.~\cite{aa}. The \textit{enhanced power graph} $\mathcal{P}_e(G)$ of a group $G$ is the simple graph with vertices being the elements of $G$ and two distinct vertices $x,y\in G$ are adjacent in $\mathcal{P}_e(G)$ if $\langle x,y\rangle$ is a cyclic subgroup of $G$. Since then, several combinatorial properties, e.g.~automorphism \cite{zahirovic}, connectivity \cite{bera2, bera4, bera}, metric dimension \cite{ma2} etc.~of the enhanced power graph have also been explored in the literature. For a deeper dive into the literature of the enhanced power graphs of groups, we refer to the survey paper \cite{make} by Ma et al.~and the references therein.

It is easy to see that for any finite group $G$, $\mathcal{P}(G)$ is a spanning subgraph of $\mathcal{P}_e(G)$ and $\mathcal{P}_e(G)$ is a spanning subgraph of $\Delta(G)$. In the study of the hierarchy of graphs defined over finite groups, Cameron and Kuzma \cite{deepc} introduced a new notion of graph called the \textit{deep commuting graph}, which lies between the enhanced power graph and the commuting graph. This article is a comprehensive study of the deep commuting graph defined over finite groups.

\subsection{Organization of the article}

Over the next section, we go through some graph-theoretic and group-theoretic preliminaries. We conclude the section with a formal definition and some already investigated properties of the deep commuting graph of finite groups. In Section \ref{Some properties of the deep commuting graph}, we explore several properties of the deep commuting graph, e.g.~completeness, Eularianess, universality, perfectness, connectivity of reduced graph and equality with other graphs. In Section \ref{Deep commuting graph of finite abelian groups}, we study the deep commuting graphs of finite abelian groups. We classify finite abelian groups where the deep commuting graph coincides with the commuting graph or the enhanced power graph. We also study dominant vertices and investigate the cases where the reduced graph remains connected. In Section \ref{Deep commuting graph of some non-abelian groups}, we study the aforementioned properties of the deep commuting graphs of dihedral groups, generalized quaternion groups, Heisenberg groups, symmetric and alternating groups. We also classify finite symmetric and alternating groups, for which the deep commuting graph is perfect. We conclude this article with a few open ends arising from our study of the deep commuting graph of finite groups.

\section{Preliminaries}\label{Preliminaries}

Throughout this article, all our graphs are finite, simple and undirected, unless otherwise mentioned. Let $\Gamma$ be a graph with vertex set $V(\Gamma)$ and edge set $E(\Gamma)$. For $u,v\in V(\Gamma)$, we say $u$ is \emph{adjacent} to $v$ if there is an edge between $u$ and $v$, and we denote it by $u \sim v$.  A graph $\mathcal{S}=(V(\mathcal{S}),E(\mathcal{S}))$ is called a \textit{subgraph} of $\Gamma$ if $V(\mathcal{S})\subseteq V(\Gamma)$ and $E(\mathcal{S})\subseteq E(\Gamma)$. For a subset $S$ of $V(\Gamma)$, the subgraph induced in $\Gamma$ by $S$, or the \textit{induced subgraph} $\Gamma\vert_S$ is defined as the graph with vertex set $S$, and two vertices are adjacent in $\Gamma\vert_S$ if and only if they are adjacent in $\Gamma$. A subgraph $\mathcal{S}$ of $\Gamma$ is called a \textit{spanning subgraph} of $\Gamma$ if $V(\mathcal{S})=V(\Gamma)$. The \emph{neighborhood} $\mathcal{N}_{\Gamma}(v)$ of a vertex $v$ is the set of all vertices adjacent to $v$ in $\Gamma$. By $\mathcal{N}_{\Gamma}[v]$, we mean $\mathcal{N}_{\Gamma}[v]=\mathcal{N}_{\Gamma}(v)\cup \{v\}$. A vertex $v$ is called a \emph{dominant vertex} of $\Gamma$, if $\mathcal{N}_{\Gamma}[v]=V(\Gamma)$. We denote the set of dominant vertices of $\Gamma$ by $_d\Gamma$. We say $\Gamma$ is complete if $_d\Gamma =V(\Gamma)$. The \textit{reduced graph} of $\Gamma$, denoted by $\Gamma^*$ is defined to be the induced subgraph of $\Gamma\vert_{V(\Gamma)\setminus{_d}\Gamma}$. Note that, if $\Gamma$ is a complete graph, then $\Gamma^*$ is the empty graph.

For a connected graph $\Gamma$, the \emph{distance} between vertices $u$ and $v$, denoted by $d(u, v)$, is the length of a shortest path from $u$ to $v$. The \emph{diameter} of $\Gamma$, denoted by ${\rm diam}(\Gamma)$, is defined as ${\rm diam}(\Gamma)=\max\{d(u,v)|u,v\in V(\Gamma)\}.$ The \textit{strong product} of two graphs $\Gamma$ and $\mathcal{H}$, denoted by $\Gamma\boxtimes\mathcal{H}$, is defined to be the graph with vertex set being the Cartesian product $V(\Gamma)\times V(\mathcal{H})$ and two distinct vertices $(u_1,u_2)$ and $(v_1,v_2)$ are adjacent in $\Gamma\boxtimes\mathcal{H}$ if either $u_1=v_1$ and $u_2\sim v_2$ in $\mathcal{H}$ or $u_2=v_2$ and $u_1\sim v_1$ in $\Gamma$ or $u_1\sim v_1$ in $\Gamma$ and $u_2\sim v_2$ in $\mathcal{H}$. 

The \emph{clique number} $\omega(\Gamma)$ of a graph $\Gamma$ is the largest possible size of a clique inside the graph. The \emph{chromatic number} $\chi(\Gamma)$ of $\Gamma$ is the minimum number of colors required to color the vertices of the graph, so that no two adjacent vertices receive the same color. It is easy to see that for any finite graph $\Gamma$, $\omega(\Gamma)\leq \chi(\Gamma)$. A finite graph $\Gamma$ is said to be a \textit{perfect graph}, if $\omega(\Lambda)=\chi(\Lambda)$ for every induced subgraph $\Lambda$ of $\Gamma$.

For two graphs $\Gamma_1$ and $\Gamma_2$ with disjoint vertex sets, the \emph{join graph} $\Gamma_1 \vee \Gamma_2$ of $\Gamma_1$ and $\Gamma_2$  is the graph obtained from the union of $\Gamma_1$ and $\Gamma_2$ by adding new edges from each vertex of $\Gamma_1$ to every vertex of $\Gamma_2.$ For a graph $\Im$ with $k$ vertices $V(\Im) = \{v_1,v_2, \dots, v_k\}$, let $\Gamma_1, \Gamma_2, \dots, \Gamma_k$ be pairwise vertex disjoint graphs. The $\Im$\textit{-generalized join graph} of $\Gamma_1, \Gamma_2, \dots, \Gamma_k$, denoted by $\Im [\Gamma_1, \Gamma_2, \dots, \Gamma_k]$ is the graph formed by replacing each vertex $v_i$ of $\Im$ by the graph $\Gamma_i$ and then joining each vertex of $\Gamma_i$ to every vertex of $\Gamma_j$ whenever $v_i \sim v_j$ in $\Im$. It is easy to see that the graph $K_2[\Gamma_1,\Gamma_2]$ coincides with the usual join $\Gamma_1 \vee \Gamma_2$ of $\Gamma_1$ and $\Gamma_2$. The following result relates the perfectness of a graph with the perfectness of a generalized join of that graph.

\begin{theorem}[{\rm\cite[Theorem 7]{kavaskar}}]\label{perfectness of generalized join}
Let $\Gamma$ be a graph with $V(\Gamma)=\{v_1,v_2,\dots,v_n\}$, and $\mathcal{H}_1,\mathcal{H}_2,\dots,\mathcal{H}_n$ be graphs such that each $\mathcal{H}_i$ is either complete or a totally disconnected graph. Then $\Gamma$ is perfect if and only if $\Gamma[\mathcal{H}_1,\mathcal{H}_2,\dots,\mathcal{H}_n]$ is perfect.     
\end{theorem}

In a graph $\Gamma$, an \emph{Eulerian trail} is a trail that contains every edge of $\Gamma$. An \emph{Eulerian circuit} or \emph{Eulerian cycle} in $\Gamma$ is an Eulerian trail that starts and ends on the same vertex of $\Gamma$. A graph is called \emph{Eulerian} if it contains an Eularian circuit. It is well known that a finite simple undirected graph is Eulerian if and only if every vertex of that graph has even degree. We refer to \cite{west} for other unexplained graph theoretic terminologies used in this paper.

Throughout the article, we denote the cyclic group of order $n$ by $C_n$ unless otherwise mentioned. For $n\in\mathbb{N}$, by $[n]$ we will denote the set $\{1,2,\dots,n\}$. The \textit{exponent} of a group $G$, denoted by ${\rm exp}(G)$ is the smallest positive integer $m$ for which $g^m=e$ for all $g\in G$. If no such positive integer exists, then we assume ${\rm exp}(G)=\infty$. For a group $G$, the \textit{center} of the group $G$, denoted by $Z(G)$ is defined to be the set of all elements of $G$ which commute with every element of the group $G$. For $x,y\in G$, we denote $x^y=y^{-1}xy$, i.e. the conjugation of $x$ by $y$ in $G$. Clearly if $x\in Z(G)$, then $x^y=x$ for all $y\in G$ and vice versa. Also it is easy to check that, for $x,y,z\in G$, $(xz)^y=x^yz^y$ and $(x^y)^z=x^{yz}$. 

The \textit{centralizer} $C_{G}(S)$ of a subset $S$ of $G$ is defined as the set $C_{G}(S):=\{g\in G~\vert~gs=sg,\text{ for all } s\in S\}$. For any particular element $a\in G$, the centralizer of the element $a$ in $G$ is defined as $C_{G}(a):=C_{G}(\{a\})$. It is also not difficult to see that the centralizer $C_{G}(S)$ is a subgroup of $G$. The \textit{commutator} of $g$ and $h$, denoted by $[g,h]$ is defined as $[g,h]=g^{-1}h^{-1}gh$. Note that $[g^y,h^y]=[g,h]^y$ for any $y\in G$. The \textit{derived group} or the \textit{commutator subgroup} of $G$, denoted by $G'$ or $[G,G]$ is defined to be the subgroup generated by all the commutators of $G$. So, $G/G'$ is abelian for any group $G$.

A group $G$ is called a \textit{simple group} if $G$ is non-trivial and it has no non-trivial normal subgroup. $G$ is called a \textit{perfect group} if $G'=G$. Clearly, finite non-abelian simple groups are perfect. $G$ is called a \textit{quasisimple group} if $G$ is perfect and $G/Z(G)$ is simple. $G$ is said to be \textit{nilpotent} of class $n$ or \textit{nilpotency class} $n$ if it has a terminating lower central series, i.e. a series of normal subgroups of the form $$G=\gamma_1(G)\triangleright \gamma_2(G)\triangleright\dots\triangleright \gamma_{n+1}(G)=\{e\},$$ where $\gamma_{i+1}(G)=[\gamma_i(G),G]$ for $i\geq 1$ and $\gamma_n(G)$ is non trivial. Clearly, $\gamma_2(G)=G'=[G,G]$. Also, a group $G$ has nilpotency class $1$ if and only if it is abelian and it is of nilpotency class $2$ if and only if $G'\subseteq Z(G)$. For a group $G$ and an abelian group $H$, $H.G$ denotes a group with $Z(H.G)\cong H$ and $H.G/Z(H.G)\cong G$. A group $G$ is called a \textit{metacyclic group} if $G$ has a cyclic normal subgroup $N$ such that the quotient group $G/N$ is also cyclic.

Let $p$ be a prime number. The following groups will be ubiquitous throughout the article:
\begin{enumerate}[\rm (i)]\label{definition of some groups}

\item \emph{Elementary abelian $p$-group:} All the non identity elements of the group have order $p$.

\item \emph{Dihedral group:} $D_{2n} = \langle a, b ~\lvert~ a^{n} = b^2 = e, ~ ab = ba^{-1} \rangle,~n\geq 3.$

\item \emph{Generalized quaternion group:} $Q_{4n} = \langle a, b ~\lvert~ a^{2n} = e, ~ b^2 = a^n, ~ ab = ba^{-1} \rangle,~ n\geq 2.$

\item \emph{Heisenberg group:}
\\
$
\begin{aligned}
	H_3(\mathbb{Z}/p^k\mathbb{Z})=\langle y_1,y_2,w~\vert~[y_1,y_2]=w,~y_1^{p^k}=y_2^{p^k}=w^{p^k}=e&=[y_1,w]=[y_2,w]\rangle,\\
	&k\in \mathbb{N}~\text{and}~p~\text{is odd}.
\end{aligned}
$
\end{enumerate} 

Our next result follows from the universal commuting identities $[ab,c]=[a,c]^b[b,c]$ and $[a,bc]=[a,c][a,b]^c$.
\begin{lemma}\label{commutator_in_center}
 Let $G$ be a group with $a,b,c\in G$. Then the following hold:
 \begin{enumerate}[\rm (i)]
     \item If $[a,c]\in Z(G)$, then $[ab,c]=[a,c][b,c]$,
     \item If $[a,b]\in Z(G)$, then $[a,bc]=[a,b][a,c]$.
 \end{enumerate}    
\end{lemma}
As a simple corollary of the above result, we can say the following.
\begin{corollary}\label{nilpo_class_2}
Let $G$ be a group with nilpotency class $2$. Then $[xy,z]=[x,z][y,z]$ and $[x,yz]=[x,y][x,z]$ for all $x,y,z \in G$.
\end{corollary}

A sequence,  
\[
\begin{tikzcd}
  G_0 \arrow[r,"f_1"] & G_1 \arrow[r, "f_2"] & G_2 \arrow[r, "f_3"] & \dots \arrow[r, "f_n"] & G_n 
\end{tikzcd}
\]
of groups and group homomorphisms is said to be \textit{exact} at $G_i$ if im$(f_i)=$ ker$(f_{i+1})$. Such sequence is said to be an \textit{exact sequence} if it is exact at each $G_i$ for all $i\in \{1,2,\dots,n-1\}$. A \textit{short exact sequence} is an exact sequence of the form
\[
\begin{tikzcd}
  \{e\} \arrow[r] & A \arrow[r, "f"] & B \arrow[r, "g"] & C \arrow[r] & \{e\}. 
\end{tikzcd}
\]

A group extension is the way to describe a group in terms of the quotient group of a larger group and a particular normal subgroup. If $G$ and $N$ are two groups, then $\hat{G}$ is called an \textit{extension} of $G$ by $N$ if there is a short exact sequence of the form
\[
\begin{tikzcd}
  \{e\} \arrow[r] & N \arrow[r, "\iota"] & \hat{G} \arrow[r, "\pi"] & G \arrow[r] & \{e\}. 
\end{tikzcd}
\]
In other words, $\hat{G}$ is an extension of the group $G$ by $N$ if $\iota$ is an injective homomorphism from $N$ to $\hat{G}$ such that $\iota(N)$ is normal in $G$ and $\hat{G}/\iota(N)$ is isomorphic to $G$. Such an extension is said to be a \textit{central extension} of $G$ by $N$, if $\iota(N)\subseteq Z(\hat{G})$.

A central extension $\hat{G}$ is called a \textit{stem extension} of $G$ by $N$, if $\iota(N)\subseteq Z(\hat{G})\cap \hat{G}^{'}$. For every finite group $G$, there exists an upper bound on the size of such stem extensions (see Theorems 2.1.4 and 2.1.22 in \cite{karpi}). A stem extension of maximal order is called a \textit{Schur cover} (or covering group) of $G$ and is denoted by $\tilde{G}$. For a finite group $G$ and for every Schur cover $\tilde{G}$ of $G$, the kernels are isomorphic. This kernel is therefore an invariant of $G$, called the Schur multiplier of $G$, and is denoted by $M(G)$. Moreover, it is a standard fact that $M(G)\cong H^2(G,\mathbb{C}^{\times})$, the second cohomology group of $G$ with coefficients in the multiplicative group $\mathbb{C}^{\times}$.



For a finite nilpotent group $G$, our next result determines the nilpotency class of a Schur cover of $G$.
\begin{lemma}\label{schur_cover_nilpotency}
Let $G$ be a finite group of nilpotency class $n$. Then $\tilde{G}$ is of nilpotency class at most $n+1$.    
\end{lemma}

\begin{proof}
Since $\tilde{G}$ is a Schur cover of $G$, the map $\pi:\tilde{G}\rightarrow G$ is surjective with ${\rm ker~} \pi=\iota(M(G))\subseteq Z(\tilde{G})$. Define the map $\pi_i:\gamma_i(\tilde{G})\rightarrow\gamma_i(G)$ to be the restriction of $\pi$ over the subset $\gamma_i(\tilde{G})$ of $\tilde{G}$. Then it is easy to see that $\pi_i$ is surjective for all $i\in[n+1]$. Now, $\pi_{n+1}(\gamma_{n+1}(\tilde{G}))=\gamma_{n+1}(G)=\{e\}$. Then $\gamma_{n+1}(\tilde{G})\subseteq {\rm ker}~\pi\subseteq Z(\tilde{G})$. Hence $\gamma_{n+2}(\tilde{G})=[\gamma_{n+1}(\tilde{G}),\tilde{G}]=\{e\}$, and the result follows.
\end{proof}

In \cite{miller}, Miller defined the \textit{non-abelian exterior square} of a group $G$. This group, denoted by $G\wedge G$, is generated by the symbols $x \wedge y,$ for $x, y \in G$, subject to the relations
\[
(xy)\wedge z=(x^y \wedge z^y)(y\wedge z), ~x\wedge (yz)= (x\wedge z)(x^z\wedge y^z), ~x\wedge x = 1.
\]
It is not difficult to see that, the homomorphism $f:G\wedge G\rightarrow[x,y]$ defined as $f(x\wedge y)=[x,y]$, is surjective. In addition, it was shown in \cite[Theorem 2]{miller} that the kernel of $f$ is isomorphic to $M(G)$. Now, denote by $M_0(G)$ the subgroup $\langle x\wedge y~|~[x, y] = 1 \rangle$ of ker $f$. Then the group $M(G)/M_0(G)$ is called the \textit{Bogomolov multiplier} of $G$ and is denoted by $B_0(G)$. It is easy to see that for a finite group $G$, $B_0(G)\cong M(G)$ if and only if they are of the same cardinality. Also, it is easy to see that if $G$ is abelian, then $G'$ is trivial, hence ker $f=G\wedge G=M_0(G)$. So, the following observation holds.
\begin{remark}\label{bgo_ab}
If $G$ is an abelian group, then the Bogomolov multiplier $B_0(G)$ is trivial.    
\end{remark}

The converse of the above observation is not true. It will be shown in Corollary \ref{Schur cover of D_2n and Q_4n} that the Schur multiplier of $Q_{4n}$ is trivial. Hence, the Bogomolov multiplier is also trivial, yet $Q_{4n}$ is non-abelian.

\subsection{The deep commuting graph}

We are now in a state to define the \textit{deep commuting graph} of a finite group. Let $G$ be a finite group and $\hat{G}$ be a central extension of $G$. Define a graph $\Gamma$ with vertices being the elements of $G$ and two vertices are adjacent if their preimages commute in every central extension of $G$. In \cite[Theorem 4.1]{deepc}, it was shown that the preimages of two elements of a finite group $G$ will commute in every central extension if and only if they commute in a Schur cover. It was also shown there that if the preimages of two elements of $G$ commute in any Schur cover of $G$, then they must commute in every Schur cover of $G$.

With the above results in hand, we can come to the following definition of the deep commuting graph of a finite group. 
\begin{definition}[\cite{deepc}]\label{deep_commuting}
Let $G$ be a finite group with $\Tilde{G}$ being a Schur cover of $G$. Then the \textit{deep commuting graph} of $G$, denoted by $\Delta_D(G)$ is a finite simple graph with vertices being the elements of $G$ and two distinct elements are adjacent if their pre images commute in $\Tilde{G}$.     
\end{definition}
In \cite{deepc}, the authors showed that the deep commuting graph fits well inside the hierarchy of graphs defined on groups. In particular they proved the following, 
\begin{theorem}[\cite{deepc}]\label{inclusion}
    Let G be a finite group. Then
    \[
    E(\mathcal{P}_e(G))\subseteq E(\Delta_D(G)) \subseteq E(\Delta(G)).
    \]
\end{theorem}

From the above discussion, again an interesting query of classifying the groups for which at least one of the inclusions become equality arises. In that context, the authors proved the following,
\begin{theorem}[\cite{deepc}]\label{equality} Let $G$ be a finite group. Then
\begin{enumerate}[\rm (i)]
    \item $\Delta_D(G)=\mathcal{P}_e(G)$ if and only if $G$ has the following property: Let $\tilde{G}$ be a Schur cover of $G$, with $\tilde{G}/M(G)=G$. Then for any subgroup $A$ of $G$, with $B$ being the corresponding preimage of $A$ in $\tilde{G}$ (so $M(G)\leq B$ and $B/M(G)=A$), if $B$ is abelian, then $A$ is cyclic.
    \item $\Delta_D(G)=\Delta(G)$ if and only if the Bogomolov multiplier of $G$ is equal to the Schur multiplier.
\end{enumerate}
\end{theorem}
As a simple application of the above theorem, one can see the following corollary.
\begin{corollary}[\cite{deepc}]\label{bogomolov_trivial}
Let $G$ be a finite group with $B_0(G)=\{e\}$. Then $\Delta_D(G)=\Delta(G)$ if and
only if the Schur multiplier of $G$ is trivial.
\end{corollary}

Also, the following two results on the deep commuting graph are interesting and useful.

\begin{lemma}[{\rm\cite[Proposition 6.1(c)]{cameron}}]\label{strong_product}
Let $G$ and $H$ be two finite groups such that ${\rm gcd}(|G|,|H|)=1$. Then $\Delta_D(G\times H)=\Delta_D(G)\boxtimes\Delta_D(H)$.   
\end{lemma}

\begin{lemma}[{\rm\cite[Theorem 9.1(c)]{cameron}}]\label{dominant_center}
Let $G$ be a finite group with $\tilde{G}$ being a Schur cover of $G$. Then $_d\Delta_D(G)=\pi(Z(\tilde{G}))$, where $\pi$ is the projection map.
\end{lemma}

Again, for any finite group $G$, it can be seen that ${_d}\mathcal{P}_e(G)\subseteq{_d}\Delta_D(G)\subseteq{_d}\Delta(G)$. So, as a simple corollary of the above lemma, the following can be seen.
\begin{corollary}\label{center_schur_cover}
    Let $G$ be a finite group such that ${_d}\mathcal{P}_e(G)={_d}\Delta(G)=\mathfrak{D}$. Then for any Schur cover $\tilde{G}$ of $G$, $Z(\tilde{G})=\pi^{-1}(\mathfrak{D})$ where $\pi$ is the projection map. Also, if ${_d}\Delta(G)=\{e\}$, then $Z(\tilde{G})\cong M(G)$.
\end{corollary}

\section{Some properties of the deep commuting graph}\label{Some properties of the deep commuting graph}

It is not difficult to see that for any group $G$ and a proper subgroup $H$ of $G$, the induced subgraphs $\mathcal{P}_e(G)\vert_H$ and $\Delta(G)\vert_H$ are same as $\mathcal{P}_e(H)$ and $\Delta(H)$, respectively. However, in contradiction to the enhanced power graph or the commuting graph, the deep commuting graph lacks this property. To see an example of the above phenomenon, we need to know about Schur cover of finite non-cyclic abelian $p$-groups.

Consider a finite abelian $p$-group $G=\prod_{1\leq i \leq k}C_{p^{r_i}}$ with $r_1\geq r_2\geq \dots \geq r_k$ and $k\geq 2$. Then by \cite[Corollary 2.2.12]{karpi}, we have $M(G)\cong\prod_{2\leq i \leq k}{C_{p^{r_i}}}^{(i-1)}$. In addition, the following result provides a Schur cover for $G$.
\begin{theorem}[{\rm\cite[Theorem 2.9.3]{karpi}}]\label{Cover_abelian_p}
For $G=\prod_{1\leq i \leq k}C_{p^{r_i}}$, with $r_1\geq r_2\geq \dots \geq r_k$, $k\geq 2$, and let $x_1, x_2, \dots, x_k$ be a collection of generators for $C_{p^{r_1}},C_{p^{r_2}},\dots,C_{p^{r_k}}$, respectively. Define the group $\tilde{G}$ by
\[
\tilde{G}=\langle \tilde{x}_i, a_{jl}~\mid~[\tilde{x}_j, \tilde{x}_l]=a_{jl},~  \tilde{x}_i^{p^{r_i}}=[\tilde{x}_i,a_{jl}]=e,~1\leq i,j,l \leq k,~j<l \rangle.
\]
Then $\tilde{G}$ is a Schur cover of $G$.
\end{theorem}
Also, it can be seen that $\tilde{G}'$ is nontrivial. Hence, $\tilde{G}$ is of nilpotency class at least two. Again, since $G$ is abelian, using Lemma \ref{schur_cover_nilpotency} we can say that $\tilde{G}$ is of nilpotency class at most two. Hence, $\tilde{G}$ is of nilpotency class two. Then using Corollary \ref{nilpo_class_2}, it can be derived that $[x_i^m,x_j^n]=[x_i,x_j]^{mn}$ for all $i,j\in [k]$ and $m,n\in\mathbb{N}$. Moreover, from the fact that $[x_j,x_i]=[x_i,x_j]^{-1}$ and Corollary \ref{nilpo_class_2}, the following relation can be derived; 
\begin{equation}\label{commutators}
\left[\prod_{i=1}^k {\tilde{x}_i}^{s_i}, \prod_{j=1}^k {\tilde{x}_j}^{t_j}\right]=\prod_{1 \leq i<j \leq k}a_{ij}^{s_it_j-s_jt_i},~~~~~o(a_{ij})=p^{r_j}.
\end{equation}

Now, consider the group $C_p\times C_p$ with $x_1,x_2$ being respective generators for the two $C_p$'s. Using Theorem \ref{Cover_abelian_p} and relation (\ref{commutators}), it can be seen that $\Delta_D(C_p\times C_p)$ consists of the $p+1$ cliques $\langle(x_1,e)\rangle,\langle(e,x_2)\rangle,\langle(x_1,x_2)\rangle,\langle(x_1,x_2^2)\rangle,\dots,\langle(x_1,x_2^{p-1})\rangle$, each order $p$, with identity being the only common vertex. But in $C_{p^2}\times C_{p^2}$, consider $y_1,y_2$ as respective generators for the two $C_{p^2}$'s. Let, $H$ be the subgroup generated by all elements of order $p$ in $C_{p^2}\times C_{p^2}$. Then clearly $H\cong C_p\times C_p$ and a non identity element of $H$ will be of the form $(y_1^{k_1p},y_2^{k_2p})$, where ${\rm gcd}(k_1,p)={\rm gcd}(k_2,p)=1$. Now, consider two non identity elements of the form $(y_1^{k_1p},y_2^{k_2p})$ and $(y_1^{l_1p},y_2^{l_2p})$ in $H$. Then the relation (\ref{commutators}) gives $[(\tilde{y_1}^{k_1p},\tilde{y_2}^{k_2p}),(\tilde{y_1}^{l_1p},\tilde{y_2}^{l_2p})]=[\tilde{y_1},\tilde{y_2}]^{k_1l_2p^2-k_2l_1p^2}=e$. Hence, $(y_1^{k_1p},y_2^{k_2p})\sim(y_1^{l_1p},y_2^{l_2p})$ in $\Delta_D(C_{p^2}\times C_{p^2})$, which implies $\Delta_D(C_{p^2}\times C_{p^2})\vert_H\cong K_{p^2}$. So, the induced subgraph $\Delta_D(C_{p^2}\times C_{p^2})\vert_H\ncong \Delta_D(H)$.

Our next result provides a sufficient condition for which the induced subgraph $\Delta_D(G)\vert_H$ for a proper subgroup $H$ of $G$, is the same as $\Delta_D(H)$. 
\begin{proposition}\label{induced_subgraph_equality}
For a finite group $G$, consider the stem extension,
\[
\begin{tikzcd}
  \{e\} \arrow[r] & M(G) \arrow[r, "\iota"] & \tilde{G} \arrow[r, "\pi"] & G \arrow[r] & \{e\}. 
\end{tikzcd}
\]
Let $H$ be a proper subgroup of $G$, such that $M(H)\cong M(G)$ and $\iota(M(G))\subseteq [\pi^{-1}(H),\pi^{-1}(H)]$. Then $\Delta_D(G)\vert_H=\Delta_D(H)$.
\end{proposition}

\begin{proof}
Since $H$ is a subgroup of $G$, consider the central extension
\[
\begin{tikzcd}
  \{e\} \arrow[r] & M(G) \arrow[r, "\iota"] & \pi^{-1}(H) \arrow[r, "\pi_{_H}"] & H \arrow[r] & \{e\}, 
\end{tikzcd}
\]
where, $\pi_{_H}$ is the restriction of $\pi$ over $\pi^{-1}(H)$. Since $\iota(M(G))\subseteq [\pi^{-1}(H),\pi^{-1}(H)]$, the above extension is a stem extension. In addition, we have $M(G)\cong M(H)$, which implies $\lvert\pi^{-1}(H)\rvert=\lvert\tilde{H}\rvert$. Hence, $\pi^{-1}(H)$ is a Schur cover of $H$. Let $h_1,h_2\in H$ with $\tilde{h_1},\tilde{h_2}$ being their respective preimages in $\pi^{-1}(H)$. Then from the above construction it easily follows that $\tilde{h_1},\tilde{h_2}$ are also preimages of $h_1, h_2$ in $\tilde{G}$, respectively. So, the preimages of $h_1$ and $h_2$ commute in $\pi^{-1}(H)$ if and only if they commute in $\tilde{G}$. Hence, the result follows.  
\end{proof}

Our next result classifies the finite groups, for which the deep commuting graph becomes complete. In this regard, we have the following:

\begin{theorem}\label{complete}
Let $G$ be a finite group. Then $\Delta_D(G)$ is complete if and only if $G$ is cyclic.
\end{theorem}
\begin{proof}
Let $G$ be a finite cyclic group. Then $\mathcal{P}_e(G)$ is complete. Hence, from Theorem \ref{inclusion} we can conclude that $\Delta_D(G)$ is also complete. 

Conversely, suppose $\Delta_D(G)$ is complete. Then from Theorem \ref{inclusion} we have $\Delta(G)$ is also complete and so $G$ is abelian. Let $G_1,G_2,\dots,G_l$ be the Sylow $p$-subgroups of $G$. Then Lemma \ref{strong_product} implies $\Delta_D(G)\cong\Delta_D(G_1)\boxtimes\Delta_D(G_2)\boxtimes\dots\boxtimes\Delta_D(G_l)$. As $\Delta_D(G)$ is complete, so $\Delta_D(G_t)$ is complete for all $t\in[l]$. Now, if $G$ is not cyclic, then there exist some $j\in[l]$ such that $G_j$ is non-cyclic. Assume $G_j$ to be of the form $G_j=\prod_{1\leq i \leq k}C_{p^{r_i}}$, with $r_1\geq r_2\geq \dots \geq r_k$, $k\geq 2$, for some prime $p$, and let $x_1, x_2, \dots, x_k$ be a collection of generators for $C_{p^{r_1}},C_{p^{r_2}},\dots,C_{p^{r_k}}$, respectively. Let $\tilde{x_1}, \tilde{x_2}, \dots, \tilde{x_k}$ be a collection of preimages of $x_1, x_2, \dots, x_k$ in $\tilde{G_j}$, respectively. Then from Theorem \ref{Cover_abelian_p}, we can say that $[\tilde{x_1},\tilde{x_2}]\neq e$, i.e.~$x_1\nsim x_2$ in $\Delta_D(G_j)$, which is contrary to the assumption that $\Delta_D(G_j)$ is complete. So, $G$ is cyclic. This concludes the proof.    
\end{proof}

\subsection{Eulerianess}
As a consequence of Theorem \ref{complete}, it can be seen that $\Delta_D(C_n)$ is Eulerian if and only if $n$ is an odd positive integer. Our next result classifies finite groups for which the deep commuting graph is Eulerian.
\begin{theorem}\label{eulerian}
Let $G$ be finite group. Then $\Delta_D(G)$ is Eulerian if and only if $\lvert G\rvert$ is odd.   
\end{theorem}
\begin{proof}
Let $G$ be a finite group such that $\Delta_D(G)$ is Eulerian. Then every vertex of $\Delta_D(G)$ has even degree. Since ${\rm deg}_{\Delta_D(G)}(e)=\lvert G\rvert-1$ we have $\lvert G\rvert-1$ even. Hence, $\lvert G\rvert$ odd.

Conversely, let $G$ be a finite group of odd order and $g\in G$. Then $N_{\Delta_D(G)}[g]=\pi(C_{\tilde{G}}(\tilde{g}))$, where $\pi$ is the projection map from $\tilde{G}$ to $G$, and $\tilde{g}\in \pi^{-1}(g)$. Since $C_{\tilde{G}}(\tilde{g})$ is a subgroup of $\tilde{G}$, we have $N_{\Delta_D(G)}[g]$ is also a subgroup of $G$. So, $\lvert N_{\Delta_D(G)}[g]\rvert$ is odd. Hence ${\rm deg}_{\Delta_D(G)}(g)=\lvert N_{\Delta_D(G)}[g]\rvert-1$ is even. This concludes the proof.    
\end{proof}

\subsection{Universality}
A family of finite, simple and undirected graphs is said to be an \emph{universal family}, if every finite, simple and undirected graph is embeddable in some graph of that family, i.e.~for any finite, simple and undirected graph $\Gamma$, there exists a graph $\Lambda$ in that family such that $\Lambda$ contains $\Gamma$ as an induced subgraph. Universality of some classes of graphs defined over finite groups have been studied in \cite{biswas, cameron}. Our next result studies the universality of the family of deep commuting graphs of finite nilpotent groups. In particular, we have the following:
\begin{theorem}\label{universality_nilpotent}
Let $\Gamma$ be a finite, simple and undirected graph. Then there exists a finite nilpotent group $G$ such that $\Delta_D(G)$ contains $\Gamma$ as an induced subgraph.    
\end{theorem}
\begin{proof}
We will proceed via induction. Let $\Gamma_1$ be the trivial graph. Clearly, the deep commuting graph of any finite nilpotent group contains $\Gamma_1$ as an induced subgraph. Now, assume that the above result holds for all graphs with $n-1$ vertices, for some positive integer $n\geq 2$. Let $\Gamma_{n}$ be a finite, simple and undirected graph with $n$ vertices. Let the vertices be denoted by $v_1,v_2,\dots,v_n$. Denote the induced subgraph $\Gamma_n\vert_{\{v_1,\dots,v_{n-1}\}}$ by $\Gamma_{n-1}$. Then $\Gamma_{n-1}$ is a finite, simple and undirected graph with $n-1$. So, by the inductive hypothesis, there exists a finite nilpotent group $H$ such that $\Delta_D(H)$ contains $\Gamma_{n-1}$ as an induced subgraph. Without loss of generality, we denote the vertices of that induced subgraph of $\Delta_D(H)$ by $v_1,v_2,\dots,v_{n-1}$ accordingly.   

Now, consider a finite non-cyclic $p$-group $K$ such that ${\rm gcd}(\lvert H\rvert,p)=1$. Then from Theorem \ref{complete}, we have $\Delta_D(K)$ is not complete. So, there exist non-identity elements $k_1,k_2\in K$ such that $k_1\nsim k_2$ in $\Delta_D(K)$. Let us denote the identity element of $H$ by $e_H$ and $K$ by $e_K$, respectively. Define the group $G:=H\times K$. Clearly, $G$ is also nilpotent. Again, from Lemma \ref{strong_product} we have $\Delta_D(G)=\Delta_D(H)\boxtimes\Delta_D(K)$. Define the subset $S$ of $G$ by $S:=(\mathcal{N}_{\Gamma_n}(v_n)\times\{e_K\})\cup((V(\Gamma_{n-1})\setminus\mathcal{N}_{\Gamma_n}(v_n))\times\{k_1\})\cup\{(e_H,k_2)\}$. We will now show that $\Gamma_n$ is isomorphic to the induced subgraph $\Delta_D(G)\vert_S$. Define the map $\phi:V(\Gamma_n)\rightarrow S$ by $\phi(v)=(v,e_K)$ for $v\in\mathcal{N}_{\Gamma_n}(v_n)$, $\phi(v)=(v,k_1)$ for $v\in V(\Gamma_{n-1})\setminus\mathcal{N}_{\Gamma_n}(v_n)$ and $\phi(v_n)=(e_H,k_2)$. Let $v_i,v_j\in V(\Gamma_n)$ such that $i,j<n$. Since $e_K\sim k_1$ in $\Delta_D(K)$, we have $\phi(v_i)\sim\phi(v_j)$ in $\Delta_D(G)\vert_{S}$ if and only if $v_i\sim v_j$ in $\Delta_D(H)$ which is same as $v_i\sim v_j$ in $\Gamma_n$. Moreover, for $i<n$, if $v_i\in\mathcal{N}_{\Gamma_n}(v_n)$, then $\phi(v_i)=(v_i,e_K)\sim(e_H,k_2)=\phi(v_n)$ in $\Delta_D(G)\vert_{S}$ and if $v\in V(\Gamma_{n-1})\setminus\mathcal{N}_{\Gamma_n}(v_n)$, then $\phi(v_i)=(v_i,k_1)\nsim(e_H,k_2)=\phi(v_n)$ in $\Delta_D(G)\vert_{S}$. So, we have $v_i\sim v_j$ in $\Gamma_n$ if and only if $\phi(v_i)\sim\phi(v_j)$ in $\Delta_D(G)\vert_{S}$, for all $v_i,v_j\in V(\Gamma_n)$. Hence, $\phi$ is an isomorphism. Thus, the result follows.
\end{proof}
It is worth noticing that the arguments in the above proof also hold if we choose both $H$ and $K$ to be abelian as well as non-abelian. With this observation in hand, we arrive at the following result.
\begin{corollary}\label{universality_corollary}
Let $\Gamma$ be a finite, simple and undirected graph. Then we have the following:
\begin{enumerate}[\rm (i)]
    \item There exists a finite abelian group $G_1$ such that $\Delta_D(G_1)$ contains $\Gamma$ as an induced subgraph;
    \item There exists a finite non-abelian group $G_2$ such that $\Delta_D(G_2)$ contains $\Gamma$ as an induced subgraph.
\end{enumerate}   
\end{corollary}

\subsection{Perfectness}\label{perfectness}
Perfectness of the graphs in the hierarchy of graphs defined on groups have been considered previously in \cite{britnell, cameron, cameron2}. In the context of deep commuting graph, it can be seen that for the finite cyclic group $C_n$, the deep commuting graph $\Delta_D(C_n)$ is complete, hence perfect. Again, it is not difficult to see that the deep commuting graph of a finite group is not always perfect. For example, using well known facts about permutation groups, and a later result (Corollary \ref{disjoint_permutations_in_An}) it is easy to see that for any $n\geq 8$, the subgraph induced in $\Delta_D(A_n)$ by the set $\{(123),(456),(178),(234),(567)\}$ is a cycle of length $5$, which is not perfect. Hence, for any $n\geq 8$, $\Delta_D(A_n)$ is not perfect. Throughout the rest of this section, we classify all finite simple groups with perfect deep commuting graphs.

In the next result, we show a relation between $\Delta(\tilde{G})$ and $\Delta_D(G)$ via generalized join of complete graphs.


\begin{lemma}\label{generalized join of the deep commuting graph}
Let $G$ be a finite group such that $\lvert M(G)\rvert=m$. Then we have $\Delta(\tilde{G})\cong\Delta_D(G)[K_m,K_m,\dots K_m]$.    
\end{lemma}

\begin{proof}
Let $G$ be a finite group with $\lvert M(G)\rvert=m$ and $g\in G$. Then $\Delta(\tilde{G})\vert_{\pi^{-1}(g)}\cong K_m$, where $\pi:\tilde{G}\rightarrow G$ is the projection map. Now define a map $\mathfrak{h}:V(\Delta(\tilde{G})) \rightarrow V(\Delta_D(G)[K_m,K_m,\dots K_m])$ which bijectively maps the vertices in $\pi^{-1}(g)$ to the vertices of $K_m$ corresponding to $g$, in $\Delta_D(G)[K_m,K_m,\dots K_m]$. Clearly $\mathfrak{h}$ is bijective. We will now show that $\mathfrak{h}$ is an isomorphism from $\Delta(\tilde{G})$ to $\Delta_D(G)[K_m,K_m,\dots K_m]$. Suppose $h_1, h_2$ are two distinct elements in $\tilde{G}$. Then, $h_1\sim h_2$ in $\Delta(\tilde{G})$ if and only if $[h_1,h_2]=e$, and $[h_1,h_2]=e$ if and only if either $\pi(h_1)=\pi(h_2)$, or $\pi(h_1)\sim\pi(h_2)$ in $\Delta_D(G)$. Then from the map $\mathfrak{h}$, it is clear that $h_1\sim h_2$ in $\Delta(\tilde{G})$ if and only if $\mathfrak{h}(h_1)\sim \mathfrak{h}(h_2)$ in $\Delta_D(G)[K_m,K_m,\dots K_m]$. Hence $\mathfrak{h}$ is an isomorphism from $\Delta(\tilde{G})$ to $\Delta_D(G)[K_m,K_m,\dots K_m]$. This concludes the proof.  
\end{proof}

As a consequence of the above lemma and Theorem \ref{perfectness of generalized join}, we can arrive at the following result.

\begin{lemma}\label{perfectness of commuting graph}
Let $G$ be a finite group. Then $\Delta_D(G)$ is perfect if and only if $\Delta(\tilde{G})$ is perfect.    
\end{lemma}

It is well known that a finite simple group is abelian if and only if it is a cyclic group of prime order, hence its deep commuting graph is perfect. In addition, we have the following simple result.
\begin{lemma}\label{Schur cover quasisimple}
The Schur cover of a finite non-abelian simple group is quasisimple.  
\end{lemma}
\begin{proof}
For a finite non-abelian simple group $G$, consider the following stem extension: 
\[
\begin{tikzcd}
  \{e\} \arrow[r] & M(G) \arrow[r, "\iota"] & \tilde{G} \arrow[r, "\pi"] & G \arrow[r] & \{e\}. 
\end{tikzcd}
\]
Since $G$ is simple and non-abelian, we have $\iota(M(G))={\rm ker }~\pi=Z(\tilde{G})$. Also, ${\rm ker}~\pi\subseteq \tilde{G}'$ and $\pi(\tilde{G}')=G'=G$, which implies $\tilde{G}'=\tilde{G}$. So $\tilde{G}$ is quasisimple.    
\end{proof}
\noindent The following result classifies all finite quasisimple groups with perfect commuting graphs.
\begin{theorem}[{\rm\cite[Theorem 1]{britnell}}]\label{perfect commuting graph}
Let $G$ be a finite non-abelian quasisimple groups. Then $\Delta(G)$ is perfect if and only if $G$ is isomorphic to one of the following groups: 
\begin{enumerate}[\rm 1)]
    \item $SL_2(q)$ with $q\geq 4$,
    \item $PSL_3(2)$,
    \item $PSL_3(4)$, $C_2.PSL_3(4)$, $C_3.PSL_3(4)$, $(C_2\times C_2).PSL_3(4)$, $C_6.PSL_3(4)$, $(C_6\times C_2).PSL_3(4)$, $(C_4\times C_4).PSL_3(4)$, $(C_{12}\times C_4).PSL_3(4)$,
    \item $A_6$, $C_3.A_6$, $C_6.A_6$,
    \item $C_6.A_7$,
    \item $Sz(2^{2a+1})$ with $a\geq 1$,
    \item $C_2.Sz(8)$, $(C_2\times C_2).Sz(8)$.
\end{enumerate}
\end{theorem}
\noindent Here $Sz(2^{2a+1})$ denote the Suzuki groups of Lie type and $PSL_{n}(q)$ denote the central quotient of the group of $n\times n$ unimodular matrices over the field $\mathbb{F}_q$, for some prime power $q$.

With the above results in hand, we now classify all finite non-abelian simple groups with perfect deep commuting graphs.

\begin{theorem}\label{Simple group perfect}
Let $G$ be a finite non-abelian simple group. Then $\Delta_D(G)$ is perfect if and only if $G$ is isomorphic to one of the following groups:
\begin{enumerate}[\rm (i)]
    \item $PSL_2(q)$, $PSL_3(2)$, $PSL_3(4)$, for all $q\geq 4$, where $q$ is a prime power;
    \item $A_5$, $A_6$, $A_7$;
    \item $Sz(2^{2a+1})$ with $a\geq 1$.
\end{enumerate}
Here, the repetitions are $PSL_2(4)\cong PSL_2(5)\cong A_5$, $PSL_3(2)\cong PSL_2(7)$, $PSL_2(9)\cong A_6$,   
\end{theorem}

\begin{proof}
Let $G$ be a finite non-abelian simple group. Then Lemma \ref{Schur cover quasisimple} implies that $\tilde{G}$ is quasisimple. Now by Lemma \ref{perfectness of commuting graph}, to prove the result, we need to classify finite non-abelian simple groups whose Schur cover is among the groups of seven types mentioned in Theorem \ref{perfect commuting graph}.

Suppose $G$ is a finite non-abelian simple group with $SL_2(q)$ being its Schur cover, for some $q\geq 4$. Then it can be seen that $G\cong SL_2(q)/Z(SL_2(q))=PSL_2(q)$. Now, it was shown in \cite[Theorem 7.1.1(ii)]{karpi} that $SL_2(q)$ is the unique Schur cover of $PSL_2(q)$, for $q\geq 5,~q\neq 9$. Also, it is well known \cite[Page-245(iv)]{karpi} that $PSL_2(4)\cong PSL_2(5)\cong A_5$, $PSL_2(9)\cong A_6$, the Schur cover of $PSL_2(5)$ is $SL_2(5)$, and the Schur cover of $A_6$ is $C_6.A_6$. Also, from \cite[Page-245(iv)]{karpi} we have $PSL_3(2)\cong PSL_2(7)$, so $SL_2(7)$ is a Schur cover of $PSL_3(2)$. Hence, the groups $PSL_2(q)$ with $q\geq 4$, $PSL_3(2)$, $A_5$ and $A_6$ are the only possible simple groups such that their Schur covers are among the groups of type $1$ and type $4$ of Theorem \ref{perfect commuting graph}. So, $\Delta_D(PSL_2(q))$ with $q\geq 4$, $\Delta_D(PSL_3(2))$, $\Delta_D(A_5)$ and $\Delta_D(A_6)$ are perfect.

Now, from \cite[Theorem 7.1.1(ii)]{karpi} it can be seen that $PSL_3(2)$ is simple, but not a self cover. So no finite group has $PSL_3(2)$ as its Schur cover. Also, among all groups in type $3$, only $(C_{12}\times C_4).PSL_3(4)$ can be a Schur cover of a simple group, and $(C_{12}\times C_4).PSL_3(4)$ is indeed the Schur cover of $PSL_3(4)$. In addition, it can be seen that $A_7$ is the only simple group with Schur cover $C_6.A_7$. So, $\Delta_D(PSL_3(4))$ and $\Delta_D(A_7)$ are perfect.

From \cite[Theorem 7.4.2]{karpi}, it can be seen that for $a\geq 2$, the simple groups $Sz(2^{2a+1})$ are self covers, and for $a=1$, $(C_2\times C_2).Sz(8)$ is a Schur cover of the simple group $Sz(8)$. So, $Sz(2^{2a+1})$ with $a\geq 1$ are the only simple groups such that their Schur covers are among the groups of type $6$ and type $7$ of Theorem \ref{perfect commuting graph}. Hence, $\Delta_D(Sz(2^{2a+1}))$ with $a\geq 1$ are perfect. This concludes the proof.   
\end{proof}

We have also considered perfectness of the deep commuting graph of finite symmetric and alternating groups later in Theorem \ref{permutation_perfect}.

\subsection{Connectivity of the reduced graph}

As a consequence of Theorem \ref{complete}, it can be seen that for any non-cyclic finite group $G$, the reduced graph $\Delta_D(G)^*$ is always non empty. Our next result studies the connectedness of $\Delta_D(G)^*$ for a finite non-cyclic nilpotent group $G$.

\begin{theorem}\label{connectivity_nilpotent}
Let $G$ be a finite non-cyclic nilpotent group with $G_1, G_2, \dots, G_m$ being it's distinct Sylow subgroups. Then we have the following:
\begin{enumerate}[\rm (i)]
    \item If at least two of the $G_i$-s are non-cyclic, then $\Delta_D(G)^*$ is connected;
    \item If $G_i$ is non-cyclic for exactly one $i$, then $\Delta_D(G)^*$ is connected if and only if $\Delta_D(G_i)^*$ is connected.
\end{enumerate}   
\end{theorem}

\begin{proof}
Let $G$ be a finite non-cyclic nilpotent group such that $G=G_1G_2\cdots G_m$, where $G_i$ is the Sylow $p_i$-subgroup of $G$ for $i\in [m]$. Here $G=\{x_1x_2\cdots x_m~\vert~x_i\in G_i,~1\leq i\leq m\}$. Then from Lemma \ref{strong_product}, we have $\Delta_D(G)\cong\Delta_D(G_1)\boxtimes\Delta_D(G_2)\boxtimes\dots\boxtimes\Delta_D(G_m)$ and ${_d}\Delta_D(G)=\{z_1z_2\cdots z_m~\vert~z_i\in {_d}\Delta_D(G_i),~1\leq i\leq m\}$. For any $\bf{x}\in G$, by $(\bf{x})_j$ we denote the $j$-th part of $\bf{x}$, where $j\in [m]$.
\begin{enumerate}[\rm (i)]
    \item Let at least two of the $G_i$-s be non-cyclic. Let $\mathfrak{g},\mathfrak{h}\in G\setminus{_d}\Delta_D(G)$ be two distinct vertices of $\Delta_D(G)^*$. Then there exist $l_1,l_2\in [m]$ such that $(\mathfrak{g})_{l_1}\in G_{l_1}\setminus{_d}\Delta_D(G_{l_1})$ and $(\mathfrak{h})_{l_2}\in G_{l_2}\setminus{_d}\Delta_D(G_{l_2})$.
    
    If $l_1\neq l_2$, consider two elements $\mathfrak{u},\mathfrak{v}\in G\setminus{_d}\Delta_D(G)$ such that $(\mathfrak{u})_{l_1}=(\mathfrak{g})_{l_1}$ and $(\mathfrak{u})_j=e$, for all $j\in [m]\setminus l_1$, and $(\mathfrak{v})_{l_2}=(\mathfrak{h})_{l_2}$ and $(\mathfrak{v})_j=e$, for all $j\in [m]\setminus l_2$. Then $\mathfrak{g}\sim \mathfrak{u}\sim \mathfrak{v}\sim\mathfrak{h}$ forms a path from $\mathfrak{g}$ to $\mathfrak{h}$.

    If $l_1=l_2$, then there exists $l_3\in [m]$ such that $l_3\neq l_1$ and $G_{l_3}$ is non-cyclic. Let $a$ be an element from $G_{l_3}\setminus{_d}\Delta_D(G_{l_3})$. Consider the element $\mathfrak{w}\in G\setminus{_d}\Delta_D(G)$ such that $(\mathfrak{w})_{l_3}=a$ and $(\mathfrak{w})_j=e$, for all $j\in [m]\setminus l_3$. Then $\mathfrak{g}\sim \mathfrak{u}\sim \mathfrak{w}\sim \mathfrak{v}\sim\mathfrak{h}$ forms a path from $\mathfrak{g}$ to $\mathfrak{h}$. Hence, $\Delta_D(G)^*$ is connected.
    
    \item Let $G_i$ be non-cyclic for exactly one $i$. Without loss of generality, we assume $i=1$. Then $\Delta_D(G)^*\cong\Delta_D(G_1)^*\boxtimes\Delta_D(G_2)\boxtimes\dots\boxtimes\Delta_D(G_m)$, where $\Delta_D(G_j)$ are complete graphs for $2\leq j\leq m$. Let, $\mathfrak{g},\mathfrak{h}$ be two distinct vertices of $\Delta_D(G)^*$. Then for any path $\mathfrak{g}\sim\mathfrak{k}_1\sim\dots\sim\mathfrak{k}_t\sim\mathfrak{h}$ in $\Delta_D(G)^*$, we have a path $(\mathfrak{g})_1\sim(\mathfrak{k}_1)_1\sim\dots\sim(\mathfrak{k}_t)_1\sim(\mathfrak{h})_1$ from $(\mathfrak{g})_1$ to $(\mathfrak{h})_1$ in $\Delta_D(G_1)^*$, and vice versa. So, two arbitrary vertices $\mathfrak{g},\mathfrak{h}\in V(\Delta_D(G)^*)$ are connected in $\Delta_D(G)^*$ if and only if the corresponding vertices $(\mathfrak{g})_1$ and $(\mathfrak{h})_1$ are connected in $\Delta_D(G_1)^*$. Thus, $\Delta_D(G)^*$ is connected if and only if $\Delta_D(G_1)^*$ is connected. This concludes the proof.
\end{enumerate}
\end{proof}

\subsection{Equality with the enhanced power graph for finite $p$-groups}

It is easy to see that for any finite group $G$ with trivial Schur multiplier, we have $\Delta_D(G)=\Delta(G)$. However, \cite[Example 3]{deepc} shows that the converse is not true. Also, from \cite[Theorem 30]{aa}, it can be seen that for a finite $p$-group $G$ that is neither cyclic nor isomorphic to $Q_{2^n}$, we have $\mathcal{P}_e(G)\neq\Delta(G)$. So, for any finite $p$-group $G$ that is neither cyclic nor isomorphic to $Q_{2^n}$ with $\mathcal{P}_e(G)=\Delta_D(G)$, we have $M(G)$ is nontrivial. Moreover, a group $G$ is said to be a \textit{capable} group, if $G\cong H/Z(H)$ for some group $H$ with nontrivial center. Now for any $p$-group $G$, $\mathcal{P}_e(G)=\mathcal{P}(G)$ \cite[Theorem 28]{aa} and if is neither cyclic nor isomorphic to $Q_{2^n}$, we have $_d\mathcal{P}_e(G)={_d}\mathcal{P}(G)=\{e\}$ \cite[Proposition 4]{cameron 2}. Then the following lemma is a straightforward application of Corollary \ref{center_schur_cover}. 

\begin{lemma}\label{capable}
Suppose $G$ is a finite non-cyclic $p$-group such that $G$ is not isomorphic to $Q_{2^n}$. If $\mathcal{P}_e(G)=\Delta_D(G)$, then $G$ is capable.
\end{lemma}

Our next result provides a necessary condition for which $\Delta_D(G)$ and $\mathcal{P}_e(G)$ coincide, for any finite non-cyclic $p$-group $G$.

\begin{proposition}\label{noncyclic_commutator}
Suppose $G$ is a finite non-cyclic $p$-group such that $G$ has no subgroup isomorphic to $Q_{2^n}$. Assume that any of the following hold:  
\begin{enumerate}[\rm (i)]
    \item $G'$ is non-cyclic.
    \item There exists an element $g\in G$ with $o(g)>p$ such that $\langle g\rangle\cap G'=\{e\}$. 
\end{enumerate}
Then $E(\mathcal{P}_e(G))\subsetneq E(\Delta_D(G))$.
\end{proposition}
\begin{proof} 
Let $G$ be a finite non-cyclic $p$-group of nilpotency class $c$.
\begin{enumerate}[\rm (i)]
    \item Let $G'$ be non-cyclic. Choose a non-trivial element $g_1$ from $\gamma_c(G)$. Then $g_1\in \gamma_c(G)\subseteq \gamma_2(G)$. Since $G'=\gamma_2(G)$ is neither cyclic nor isomorphic to $Q_{2^n}$, $g_1$ is not a dominant vertex of $\mathcal{P}_e(\gamma_2(G))$. Also, we have $\mathcal{P}_e(\gamma_2(G))=\mathcal{P}_e(G)\vert_{\gamma_2(G)}$. Hence, there exists $g_2\in\gamma_2(G)$ such that $g_1\nsim g_2$ in $\mathcal{P}_e(G)$. Let $\tilde{g_1}$ and $\tilde{g_2}$ be preimages of $g_1$ and $g_2$ in $\gamma_{c}(\tilde{G})$ and $\gamma_{2}(\tilde{G})$ respectively. Then, from \cite[Lemma 3.2.6]{karpi} we can see that $[\tilde{g_1},\tilde{g_2}]\in \gamma_{c+2}(\tilde{G})$. Now, from Lemma \ref{schur_cover_nilpotency} we have $\gamma_{c+2}(\tilde{G})=\{e\}$. So, $g_1\sim g_2$ in $\Delta_D(G)$ but not in $\mathcal{P}_e(G)$. 

    \item Suppose $G$ contains an element $g$ with $o(g)>p$ and $\langle g\rangle\cap G'=\{e\}$. Let $z\in\gamma_{c}(G)$ such that $o(z)=p$. Then, $z\in\gamma_{c}(G)\subseteq\gamma_2(G)$, which implies $\langle z\rangle\cap\langle g\rangle=\{e\}$. Hence, $g^{p}\nsim z$ in $\mathcal{P}(G)=\mathcal{P}_e(G)$. Let $\tilde{g}$ and $\tilde{z}$ be preimages of $g$ and $z$ in $\tilde{G}$ and $\gamma_{c}(\tilde{G})$ respectively. Then, $\tilde{z}^p\in \text{ker }\pi\subseteq Z(\tilde{G})$. Again, from Lemma \ref{schur_cover_nilpotency} we have $[\tilde{z},\tilde{g}]\in \gamma_{c+1}(\tilde{G})\subseteq Z(\tilde{G})$. Then, from Proposition \ref{commutator_in_center}, we have $[\tilde{z},\tilde{g}^p]=[\tilde{z}^p,\tilde{g}]=e$. So, $g^{p}\sim z$ in $\Delta_D(G)$. 
\end{enumerate}
Hence, the result follows.
\end{proof}
It is worth noticing that the converse of the above statement is not true. For example, consider the following $p$-group of order $p^3$, for an odd prime $p$:
\begin{equation}\label{James p group}
 H:=\langle \alpha_1,\alpha_2,\alpha~\mid~[\alpha_1,\alpha_2]=\alpha,~\alpha_1^p=\alpha,~\alpha_2^p=\alpha^p=[\alpha_2,\alpha]=e\rangle.   
\end{equation}
Here $H$ is a self cover (see \cite[Theorem 3.3.6]{karpi}), hence $\Delta_D(H)=\Delta(H)$. Also, using a simple inductive argument, it can be seen that for any element $\alpha_1^{a_1}\alpha_2^{a_2}\alpha^{a}\in H$, we have $(\alpha_1^{a_1}\alpha_2^{a_2}\alpha^{a})^l=\alpha_1^{la_1}\alpha_2^{la_2}\alpha^{la-\frac{l(l-1)}{2}a_1a_2}$, for any $l\in \mathbb{N}$. In particular, we have $(\alpha_1^{a_1}\alpha_2^{a_2}\alpha^{a})^p=\alpha^{a_1}$. So, it is easy to see that $H$ does not satisfy any of the given conditions in Proposition \ref{noncyclic_commutator}, yet $\alpha_2\sim \alpha$ in $\Delta_D(H)$ and $\alpha_2\nsim \alpha$ in $\mathcal{P}_e(H)$.

A finite $p$-group $G$ is called \textit{extraspecial}, if $G'=Z(G)\cong C_p$ and $G/G'$ is an elementary abelian $p$-group. It is well known that $H_3(\mathbb{Z}/p\mathbb{Z})$ and the above group $H$ are the only two non abelian groups of order $p^3$, for an odd prime $p$. In addition, \cite[Corollary 8.2]{beyl} shows that an extraspecial group $G$ is capable if and only if $G\cong D_8$ or $\lvert G\rvert=p^3$ with ${\rm exp}(G)=p$, for odd prime $p$. Also, from \cite[Theorem 30]{aa} it can be seen that $\Delta_D(Q_8)=\mathcal{P}_{e}(Q_8)$. Moreover, it will be shown in Theorem \ref{dihedral equality} that $\Delta_D(D_{2n})=\mathcal{P}_e(D_{2n})$ for any even positive integer $n$, and in Theorem \ref{heisenberg_equality} that $\Delta_D(H_3(\mathbb{Z}/p\mathbb{Z}))=\mathcal{P}_{e}(H_3(\mathbb{Z}/p\mathbb{Z}))$, for any odd prime $p$. So using Lemma \ref{capable}, we now arrive at the following classification of finite extraspecial groups, for which the deep commuting graph coincides with the enhanced power graph.

\begin{corollary}\label{extraspecial}
Let $G$ be a finite extraspecial group. Then $\Delta_D(G)=\mathcal{P}_e(G)$ if and only if $G$ is isomorphic to $D_8$ or $Q_8$ or $H_3(\mathbb{Z}/p\mathbb{Z})$, where $p$ is an odd prime.    
\end{corollary}

\section{Deep commuting graph of finite abelian groups}\label{Deep commuting graph of finite abelian groups}

In this section, we study the deep commuting graph of finite abelian groups. For finite cyclic groups, Theorem \ref{complete} yields that the deep commuting graph is always complete. In case of non-cyclic finite abelian groups, Lemma \ref{strong_product} yields that the deep commuting graph of a finite non-cyclic abelian group of non prime power order is the strong product of the deep commuting graph of its Sylow $p$-subgroups. So, in order to study the deep commuting graph of finite non-cyclic abelian groups, it is enough to study the deep commuting graph of finite non-cyclic abelian $p$-groups.

Throughout this section, a finite non-cyclic abelian $p$-group $G$ is of the form $G=\prod_{1\leq i \leq k}C_{p^{r_i}}$, with $r_1\geq r_2\geq \dots \geq r_k$ and $k\geq 2$. Our next result compares the deep commuting graph and the enhanced power graph of non-cyclic abelian $p$-groups. 
\begin{theorem}\label{elementary_abelian_p}
Let $G$ be a finite non-cyclic abelian $p$-group. Then $\Delta_D(G)=\mathcal{P}_e(G)$ if and only if $G$ is elementary abelian.
\end{theorem}
\begin{proof}
$G$ is a finite non-cyclic abelian $p$-group with $x_1, x_2, \dots, x_k$ being a collection of generators for $C_{p^{r_1}},C_{p^{r_2}},\dots,C_{p^{r_k}}$ and $\Tilde{x_1},\Tilde{x_2},\dots,\Tilde{x_k}$ being their preimages in the Schur cover $\Tilde{G}$, respectively. 

When $r_1\geq 2$, we have $o(x_1)>p$ and $\langle x_1\rangle\cap G'=\{e\}$. Hence, from Proposition \ref{noncyclic_commutator}(ii) we can conclude that $E(\Delta_D(G))\subsetneq E(\mathcal{P}_e(G))$. So $\Delta_D(G)=\mathcal{P}_e(G)$ implies $r_1=1$, i.e.~$G$ is elementary abelian.

Conversely, let $G$ is elementary abelian, i.e.~$r_1=r_2= \dots =r_k=1$. Let $s$ and $t$ be two distinct non-identity elements of $G$ such that $s$ and $t$ are adjacent in $\Delta_D(G)$. Let, $s=\prod_{i=1}^kx_i^{s_i},~0\leq s_i<p$ and $t=\prod_{j=1}^kx_j^{t_j},~0\leq t_j<p$. Define the set $\mathcal{S}:=\{l\in[k]~|~s_l\neq 0~\text{or}~t_l\neq 0\}$. Then 
\[
\left[\prod_{i=1}^k\Tilde{x_i}^{s_i},\prod_{j=1}^k\Tilde{x_j}^{t_j}\right]
=\left[\prod_{i\in\mathcal{S}}\Tilde{x_i}^{s_i},\prod_{j\in\mathcal{S}}\Tilde{x_j}^{t_j}\right]
=\prod_{\substack{i,j\in\mathcal{S},\\ i<j}}a_{ij}^{s_it_j-s_jt_i}=1.
\]
Now, let $i\in[k]$ such that $t_i=0$. Choose, $u\in[k]$ such that $t_u\neq 0$. Assume $s_i\neq 0$. If $i<u$, we have $s_it_u-s_ut_i\equiv 0~({\rm mod}~p)$ and if $u<i$, we have $s_ut_i-s_it_u\equiv 0~({\rm mod}~p)$. Hence, in both the cases we have $s_it_u\equiv 0~({\rm mod}~p)$ which is contradictory. Hence, $s_i=0$. Similarly, for $j\in[k]$ it can be shown that if $s_j=0$ then $t_j=0$. So, if $i\in \mathcal{S}$, then both $s_i$ and $t_i$ are nonzero.  

Let $i,j\in\mathcal{S}$. Then $\begin{aligned}[t]
s_it_j-s_jt_i &\equiv 0~({\rm mod}~p)\\
\Rightarrow s_it_j &\equiv s_jt_i~({\rm mod}~p)\\
\Rightarrow s_it_i^{-1} &\equiv s_jt_j^{-1}~({\rm mod}~p).\\
\end{aligned}
$\\
Since, $i,j$ are arbitrary elements of $\mathcal{S}$ and every element of $G$ is of order $p$, we can conclude that $s\in \langle t\rangle$, hence $s$ and $t$ are adjacent in $\mathcal{P}_e(G)$. Since, $\mathcal{P}_e(G)$ is a spanning subgraph of $\Delta_D(G)$, we can conclude that if $G$ is a finite elementary abelian $p$-group, then  $\mathcal{P}_e(G)=\Delta_D(G)$.   
\end{proof}
On the other hand, result similar to Lemma \ref{strong_product} holds for enhanced power graphs too. In particular, as a simple consequence of Lemma $2.1$ of \cite{zahirovic}, the following can be seen. 
\begin{lemma}\label{tensor_ep}
Let $H$ and $K$ be two finite groups such that ${\rm gcd}(|H|,|K|)=1$. Then $\mathcal{P}_e(H\times K)$=$\mathcal{P}_e(H)\boxtimes\mathcal{P}_e(K)$.   
\end{lemma}
Now, with the above result, Theorem \ref{elementary_abelian_p} and Lemma \ref{strong_product}, the following corollary can be obtained for finite abelian groups. 
\begin{corollary}
Let $A$ be a finite abelian group. Then $\Delta_D(A)=\mathcal{P}_e(A)$ if and only if all the Sylow $p$-subgroups of $A$ are cyclic or elementary abelian. 
\end{corollary}
Our next result identifies the dominant vertices of $\Delta_D(G)$ for a finite non-cyclic abelian $p$-group $G$.
\begin{theorem}\label{dominant_abelian_p}
Let $G$ be a finite non-cyclic abelian $p$-group. Then $_d\Delta_D(G)=\langle x_1^{p^{r_2}}\rangle$, where  $x_1$  is a generator of $C_{p^{r_1}}$.
\end{theorem}
\begin{proof}
Let $G$ be a non-cyclic abelian $p$-group. Let $x_1,x_2,\dots, x_k$ be a collection of generators for $C_{p^{r_1}},C_{p^{r_2}},\dots,C_{p^{r_k}}$ with $\tilde{x_1},\tilde{x_2},\dots, \tilde{x_k}$ being their preimages in $\tilde{G}$, respectively. Let $x\in\langle x_1^{p^{r_2}}\rangle$. Then $x=x_1^{\beta_1}$ with $0\leq\beta_1<p^{r_1}$, such that $p^{r_2}$ divides $\beta_1$. So $[\tilde{x},\tilde{x_i}]=[\tilde{x_1},\tilde{x_i}]^{\beta_1}=e$ for all $i\in[k]$, since $r_2\geq r_3\geq\dots\geq r_k$. Hence, $x\in{_d}\Delta_D(G)$ i.e. $\langle x_1^{p^{r_2}}\rangle\subseteq{_d}\Delta_D(G)$. 

Conversely, let $g\in {_d}\Delta_D(G)$. Let $g=x_1^{\alpha_1}x_2^{\alpha_2}\dots x_k^{\alpha_k}$ with $0\leq\alpha_i<p^{r_i}$. Then $g\sim x_1$, i.e. $[\tilde{x_1},\tilde{g}]=e$. Hence, $[\tilde{x_1},\tilde{g}]=\prod_{i=2}^k[\tilde{x_1},\tilde{x_i}]^{\alpha_i}=e$, i.e. $\alpha_i=0$ for all $i$ with $2\leq i\leq k$. Hence, $g=x_1^{\alpha_1}$ with $0\leq \alpha_1<p^{r_1}$. Again, since $g\in {_d}\Delta_D(G)$ we have $g\sim x_2$. So $[\tilde{g},\tilde{x_2}]=[\tilde{x_1},\tilde{x_2}]^{\alpha_1}=e$, which implies that $p^{r_2}$ divides $\alpha_1$, i.e. $g\in\langle x_1^{p^{r_2}}\rangle$. Hence, the result follows.
\end{proof}
Using Theorem \ref{dominant_abelian_p} and Lemma \ref{strong_product}, the set of dominant vertices of $\Delta_D(A)$ can be found for a finite non-cyclic abelian group $A$. Also, it is clear that $_d\Delta_D(G)=\{e\}$ if and only if $\langle x_1^{p^{r_2}}\rangle=\{e\}$. So, we have the following.
\begin{corollary}
For a finite non-cyclic abelian $p$-group $G$, we have $_d\Delta_D(G)=\{e\}$ if  and only if $r_1=r_2$.  
\end{corollary}

For a finite non-cyclic abelian $p$-group $G$, the following can be said about the connectedness of $\Delta_D(G)^*$.

\begin{theorem}\label{connectivity_abelian_p}
Let $G$ be a finite non-cyclic abelian $p$-group. Then $\Delta_D(G)^*$ is disconnected if and only if $r_2=1$.
\end{theorem}

\begin{proof}
Suppose $r_2=1$. Then $_d\Delta_D(G)=\langle x_1^p\rangle$. Denote by $E_{x_1}$ the set of generators of the subgroup $\langle x_1\rangle$ in $G$. Then $E_{x_1}\subseteq{V(\Delta_D(G)^*)}$. Let $y\in E_{x_1}$. Then $y=x_1^{\alpha}$, where ${\rm gcd}({\alpha},p)=1$. Let $g=x_1^{\beta_1}x_2^{\beta_2}\cdots x_k^{\beta_k}$, $0\leq\beta_i<p^{r_i}$ for all $i\in[k]$, be an element of $G$ such that $y\sim g$ in $\Delta_D(G)$. Then $[\tilde{y},\tilde{g}]=[\tilde{x_1}^{\alpha},\tilde{g}]=\prod_{i=2}^k[\tilde{x_1},\tilde{x_i}]^{\alpha\beta_i}=e$, i.e.~$a_{1i}^{\alpha\beta_i}=e$ for $2\leq i\leq k$. Since ${\rm gcd}({\alpha},p)=1$ and $0\leq\beta_i<p^{r_i}$, we have $\beta_i=0$ for $2\leq i\leq k$. Hence, we have $\mathcal{N}_{\Delta_D(G)}[y]=\langle x_1\rangle$. So, $\mathcal{N}_{\Delta_D(G)}[E_{x_1}]=\underset{y\in E_{x_1}}{\cup}\langle x_1\rangle=\langle x_1\rangle$. Now, since $_d\Delta_D(G)=\langle x_1^p\rangle$, we have $\mathcal{N}_{\Delta_D(G)^*}[E_{x_1}]=\langle x_1\rangle\setminus\langle x_1^p\rangle=E_{x_1}.$ So, $E_{x_1}$ and $V(\Delta_D(G)^*)\setminus E_{x_1}$ form a separation of $\Delta_D(G)^*$, and hence $\Delta_D(G)^*$ is disconnected. 

Conversely, let $r_2\geq 2$. Then $x_1^p\notin \langle x_1^{p^{r_2}}\rangle={_d}\Delta_D(G)$. Let $h\in V(\Delta_D(G)^*)$ such that $h=x_1^{\gamma_1p^{l_1}}x_2^{\gamma_2p^{l_2}}\cdots x_k^{\gamma_kp^{l_k}}$ with $0\leq\gamma_ip^{l_i}<p^{r_i}$, $0\leq l_i< r_i$ and $\gamma_i=0$ or ${\rm gcd}(\gamma_i,p)=1$ for all $i\in [k]$. 

\noindent\textbf{\underline{Case 1}}: $\gamma_i=0$ for $2\leq i\leq k$. Then $\gamma_1\neq 0$ and $l_1< r_2$. This implies $h\in \langle x_1\rangle$, hence $h= x_1^p$ or $h\sim x_1^p$. 

\noindent\textbf{\underline{Case 2}}: $\gamma_i\neq 0$ for some $i\in [k]\setminus 1$. Take $l=\underset{\substack{2\leq i\leq k,\\\gamma_i\neq 0}}{\rm max}(r_i-l_i)$. Then $l\geq 1$ and $l_i+l-1\geq r_i-1$ with equality occuring for at least one $i$. Now $h^{p^{l-1}}=x_1^{\gamma_1p^{l_1+l-1}}x_2^{\gamma_2p^{l_2+l-1}}\cdots x_k^{\gamma_kp^{l_k+l-1}}$. So, $h^{p^{l-1}}\notin \langle x_1^{p^{r_2}}\rangle={_d}\Delta_D(G)$. Also, $[\tilde{x_1}^p,\tilde{h}^{p^{l-1}}]=\prod_{i=2}^k[\tilde{x_1},\tilde{x_i}]^{\gamma_ip^{l_i+l}}=e$, since $l_i+l\geq r_i$. Hence, we have a path $h\sim h^{p^{l-1}}\sim x_1^p$ in $\Delta_D(G)^*$. So, for any two distinct elements $h_1,h_2\in V(\Delta_D(G)^*)$, we have paths from $h_1$ to $x_1^p$ as well as $h_2$ to $x_1^p$ in $\Delta_D(G)^*$. Hence, $\Delta_D(G)^*$ is connected. 
\end{proof}
With the above result and Theorem \ref{connectivity_nilpotent}, we now classify all finite non-cyclic abelian groups $A$, such that $\Delta_D(A)^*$ is connected. 

\begin{corollary}\label{connectivity_abelian}
Let $A$ be a finite non-cyclic abelian group. Then $\Delta_D(A)^*$ is disconnected if and only if $A\cong \mathfrak{G}\times \mathfrak{K}$, where $\mathfrak{G}$ is a cyclic group and $\mathfrak{K}$ is an elementary abelian $p$-group for some prime $p$.    
\end{corollary}
Also, the next result follows from the arguments used in the proof of Theorem \ref{connectivity_nilpotent} and Theorem \ref{connectivity_abelian_p}.
\begin{theorem}\label{diameter_abelian}
Let $A$ be a finite non-cyclic abelian group such that $\Delta_D(A)^*$ is connected. Then ${\rm diam}(\Delta_D(A)^*)\leq 4$.   
\end{theorem}

\section{Deep commuting graph of some non-abelian groups}\label{Deep commuting graph of some non-abelian groups}

In this section, we study the deep commuting graph of dihedral groups, generalized quaternion groups, Heisenberg groups, symmetric and alternating groups.

\subsection{Dihedral and generalized quaternion groups}
We now study the deep commuting graph of the dihedral group $D_{2n}$ and the generalized quaternion group $Q_{4n}$. The following result yields the Schur multiplier of finite metacyclic groups.
\begin{theorem}[{\rm\cite[Theorem 2.11.3]{karpi}}]\label{meta_schur}
Let $G$ be a finite metacyclic group, say \[G=\langle a,b~|~a^m=1,~b^s=a^t,~bab^{-1}=a^r\rangle,\]
where the positive integers $m,r,s$ and $t$ satisfy \[r^s\equiv1~({\rm mod}~m),~~m|t(r-1)~~\text{and}~~t|m.\] 
Then $M(G)$ is isomorphic to $C_k$, where \[k=\frac{{\rm gcd}(r-1,m)\times l}{m}\] and \[l={\rm gcd}(1+r+\dots+r^{s-1},t).\]
\end{theorem}
As a consequence of the above theorem, we can say the following about the Schur multiplier and Schur cover of the dihedral group $D_{2n}$ and the generalized quaternion group $Q_{4n}$.
\begin{corollary}\label{Schur cover of D_2n and Q_4n}
Let $n\geq 2$ be a positive integer. Then the following hold:
\begin{enumerate}[\rm (i)]
    \item $M(Q_{4n})=\{e\}$ and $Q_{4n}$ is the unique Schur cover of itself;
    \item For $n$ is odd, $M(D_{2n})=\{e\}$ and $D_{2n}$ is the unique Schur cover of itself. For $n$ is even, $M(D_{2n})=C_2$ and $D_{4n}$ is a Schur cover of $D_{2n}$.
\end{enumerate}
\end{corollary}

\begin{proof}
\noindent
    \begin{enumerate}[\rm (i)]
        \item Clearly $G=Q_{4n}$ is a metacyclic group with $m=2n$, $s=2$, $t=n$, $r=2n-1$. Hence, $l={\rm gcd}(1+r+\dots+r^{s-1},t)=n$ and ${\rm gcd}(r-1,m)=2$. So $k=1$, hence $M(Q_{4n})=\{e\}$. So, $Q_{4n}$ is the unique Schur cover of itself.
        \item Clearly $G=D_{2n}$ is a metacyclic group with $m=n$, $s=2$, $t=n$, $r=n-1$. Hence, $l={\rm gcd}(1+r+\dots+r^{s-1},t)=n$, ${\rm gcd}(r-1,m)=1$ when $n$ is odd and ${\rm gcd}(r-1,m)=2$ when $n$ is even. So, $k=1$ when $n$ is odd and $k=2$ when $n$ is even. Hence $n$ is odd, $M(D_{2n})=\{e\}$ and $D_{2n}$ is the unique Schur cover of itself. For $n$ is even, $M(D_{2n})=C_2$. Hence, a Schur cover is of order $4n$. Now, it can be seen that the following extension
        \[\begin{tikzcd}
        \{e\} \arrow[r] & C_2 \arrow[r, "\iota"] & D_{4n} \arrow[r, "\pi"] & D_{2n} \arrow[r] & \{e\}, 
        \end{tikzcd}\]
        where $\iota$ takes $C_2$ to the center of $D_{4n}$, is a stem extension. Hence, $D_{4n}$ is a Schur cover of $D_{2n}$.
    \end{enumerate}
\end{proof} 

In the above corollary, it was obtained that the generalized quaternion group $Q_{4n}$ and for an odd integer $n$, the dihedral group $D_{2n}$ are self covers. Hence, for any positive integer $n\geq 2$, $\Delta_D(Q_{4n})=\Delta(Q_{4n})$, and for $n$ is odd, $\Delta_D(D_{2n})=\Delta(D_{2n})$. For $n$ is even, we can say the following about $\Delta_D(D_{2n})$. 
\begin{theorem}\label{dihedral equality}
For any even positive integer $n$, $\Delta_D(D_{2n})=\mathcal{P}_e(D_{2n}).$   
\end{theorem}

\begin{proof}
Consider the dihedral group $D_{2n}= \langle a, b ~\lvert~ a^{n} = b^2 = e, ~ ab = ba^{-1} \rangle$. Then $\mathcal{P}_e(D_{2n})$ is the graph where the elements of $\langle a\rangle$ in $D_{2n}$ form a clique of $\mathcal{P}_e(D_{2n})$ and all other elements $a^ib$ of $D_{2n}$ are adjacent only to the vertex $e$ in $\mathcal{P}_e(D_{2n})$. Since $E(\mathcal{P}_e(D_{2n}))\subseteq E(\Delta_D(D_{2n}))$, the elements of $\langle a\rangle$ in $D_{2n}$ form a clique in $\Delta_D(D_{2n})$ also. We will now show that all other elements $a^ib$ of $D_{2n}$ are adjacent only to the vertex $e$ in $\Delta_D(D_{2n})$ also.

Recall the projection map $\pi$ from the Schur cover $D_{4n}$ onto $D_{2n}$. Let the group $D_{4n}$ be defined by $D_{4n} = \langle \tilde{a}, \tilde{b} ~\lvert~ \tilde{a}^{2n} = \tilde{b}^2 = e, ~ \tilde{a}\tilde{b} = \tilde{b}\tilde{a}^{-1} \rangle,$ where $\pi(\tilde{a})=a$ and $\pi(\tilde{b})=b$. Let, $a^ib, a^jb\in D_{2n}$ with $0\leq j<i<n$. Then $a^ib\sim a^jb$ in $\Delta_D(D_{2n})$ if and only if $[\tilde{a}^i\tilde{b}, \tilde{a}^j\tilde{b}]=e$ in $D_{4n}$. Now, $[\tilde{a}^i\tilde{b}, \tilde{a}^j\tilde{b}]=\tilde{a}^{2(i-j)}=e$ if and only if $i=j$, which is contradictory. Hence, $a^ib\nsim a^jb$ in $\Delta_D(D_{2n})$.

Again, consider $a^i, a^jb\in D_{2n}$ with $0\leq i,j<n$. Then $a^i\sim a^jb$ in $\Delta_D(D_{2n})$ if and only if $[\tilde{a}^i, \tilde{a}^j\tilde{b}]=e$ in $D_{4n}$. Now, $[\tilde{a}^i, \tilde{a}^j\tilde{b}]=\tilde{a}^{-2i}=e$ if and only if $i=0$. Then $a^i\sim a^jb$ if and only if $a^i=e$. So, we have $\mathcal{N}_{\Delta_D(D_{2n})}(a^ib)=\{e\}$ for $0\leq i<n$. Hence, the result follows.
\end{proof}

\subsection{Heisenberg group}\label{Heisenberg group}

Recall the definition (Section \ref{Preliminaries}) of the Heisenberg group $H_3(\mathbb{Z}/p^k\mathbb{Z})$ over the ring $\mathbb{Z}/p^k\mathbb{Z}$. It can be seen that for $y_1^{i_1}y_2^{i_2}w^{i_3},y_1^{j_1}y_2^{j_2}w^{j_3}\in H_3(\mathbb{Z}/p^k\mathbb{Z})$, we have $[y_1^{i_1}y_2^{i_2}w^{i_3},y_1^{j_1}y_2^{j_2}w^{j_3}]=w^{i_1j_2-i_2j_1}$ and $(y_1^{i_1}y_2^{i_2}w^{i_3})^l=y_1^{li_1}y_2^{li_2}w^{(li_3-\frac{l(l-1)}{2}i_1i_2)}$. So, $H_3(\mathbb{Z}/p^k\mathbb{Z})'=\langle w\rangle$ and $H_3(\mathbb{Z}/p^k\mathbb{Z})$ is a $p$-group of nilpotency class $2$. Also, for $k=1$, every non-trivial element of $H_3(\mathbb{Z}/p\mathbb{Z})$ is of order $p$.

Now, from \cite[Theorem 1.1]{hatui} and the paragraph after that, we can deduce that the Schur multiplier $M(H_3(\mathbb{Z}/p^k\mathbb{Z}))=C_{p^k}\times C_{p^k}$ and the following group is a Schur cover of $H_3(\mathbb{Z}/p^k\mathbb{Z})$.
\begin{equation}\label{heisenberg_schur_cover}
  \begin{aligned}\\
     \tilde{H_3(\mathbb{Z}/p^k\mathbb{Z})}:=\langle \tilde{y_1},\tilde{y_2},\tilde{w},w_1,w_2~\vert~[\tilde{y_1},&\tilde{y_2}]=\tilde{w},~[\tilde{y_1},\tilde{w}]=w_1,~[\tilde{y_2},\tilde{w}]=w_2,\\
     &\tilde{y_1}^{p^k}=\tilde{y_2}^{p^k}=\tilde{w}^{p^k}=w_1^{p^k}=w_2^{p^k}=e\rangle~\text{of class }3.\\
    \end{aligned}\\
\end{equation}
In addition, we have $\pi(\tilde{y_1})=y_1$, $\pi(\tilde{y_2})=y_2$ and $\pi(\tilde{w})=w$, where $\pi$ is the projection map from $\tilde{H_3(\mathbb{Z}/p^k\mathbb{Z})}$ to $H_3(\mathbb{Z}/p^k\mathbb{Z})$. From (\ref{heisenberg_schur_cover}), we have $[\tilde{y_1},\tilde{w}]=w_1\neq e$. So, $y_1\nsim w$ in $\Delta_D(H_3(\mathbb{Z}/p^k\mathbb{Z}))$. So, $E(\Delta_D(H_3(\mathbb{Z}/p^k\mathbb{Z})))\subsetneq E(\Delta(H_3(\mathbb{Z}/p^k\mathbb{Z})))$ for any $k\in\mathbb{N}$. We now compare the enhanced power graph and the deep commuting graph of the Heisenberg group $H_3(\mathbb{Z}/p^k\mathbb{Z})$ over the ring $\mathbb{Z}/p^k\mathbb{Z}$.
\begin{theorem}\label{heisenberg_equality}
For $k\in\mathbb{N}$, $\mathcal{P}_e(H_3(\mathbb{Z}/p^k\mathbb{Z}))=\Delta_D(H_3(\mathbb{Z}/p^k\mathbb{Z}))$ if and only if $k=1$.    
\end{theorem}
\begin{proof}
Let $k\geq 2$. Then $o(y_1)=p^k>p$. Also, we have $[y_1^m,y_2]=w^m\neq e$ for $0<m<p^k$. So, $\langle y_1\rangle\cap H_3(\mathbb{Z}/p^k\mathbb{Z})'=\{e\}$. Hence, from Proposition \ref{noncyclic_commutator}(ii) it follows that $E(\mathcal{P}_e(H_3(\mathbb{Z}/p^k\mathbb{Z})))\subsetneq E(\Delta_D(H_3(\mathbb{Z}/p^k\mathbb{Z})))$.

Now, let $k=1$ and $y_1^{i_1}y_2^{i_2}w^{i_3}\in H_3(\mathbb{Z}/p\mathbb{Z})$ with $0\leq i_1,i_2,i_3<p$ such that $y_1^{i_1}y_2^{i_2}w^{i_3}\neq e$. Let, $y_1^{j_1}y_2^{j_2}w^{j_3}$ be another element of $H_3(\mathbb{Z}/p\mathbb{Z})$. Then, $y_1^{j_1}y_2^{j_2}w^{j_3}\in C_{H_3(\mathbb{Z}/p\mathbb{Z})}(y_1^{i_1}y_2^{i_2}w^{i_3})$ if and only if $[y_1^{i_1}y_2^{i_2}w^{i_3},y_1^{j_1}y_2^{j_2}w^{j_3}]=w^{i_1j_2-i_2j_1}=e$, if and only if $i_1j_2-i_2j_1\equiv 0~({\rm mod}~p)$. 

\vspace{.3cm}
\noindent\underline{\textbf{Case (i) ($i_1, i_2\neq 0$):}} Here, 
$\begin{aligned}[t]
 i_1j_2-i_2j_1&\equiv 0~({\rm mod}~p)\\
 \Rightarrow i_1j_2&\equiv i_2j_1~({\rm mod}~p)\\
 \Rightarrow i_1^{-1}j_1&\equiv i_2^{-1}j_2~({\rm mod}~p).\\
\end{aligned}$\\
Hence, \[C_{H_3(\mathbb{Z}/p\mathbb{Z})}(y_1^{i_1}y_2^{i_2}w^{i_3})=\langle y_1^{i_1}y_2^{i_2}w^{i_3}\rangle\times\langle w\rangle.\] Clearly, $\langle y_1^{i_1}y_2^{i_2}w^{i_3}\rangle=\mathcal{N}_{\mathcal{P}_e(H_3(\mathbb{Z}/p\mathbb{Z}))}[y_1^{i_1}y_2^{i_2}w^{i_3}]\subseteq \mathcal{N}_{\Delta_D(H_3(\mathbb{Z}/p\mathbb{Z}))}[y_1^{i_1}y_2^{i_2}w^{i_3}]$. Let $(y_1^{i_1}y_2^{i_2})^lw^{j_3}$ be an element of $C_{H_3(\mathbb{Z}/p\mathbb{Z})}(y_1^{i_1}y_2^{i_2}w^{i_3})$ such that $(y_1^{i_1}y_2^{i_2})^lw^{j_3}\sim y_1^{i_1}y_2^{i_2}w^{i_3}$ in $\Delta_D(H_3(\mathbb{Z}/p\mathbb{Z}))$. Then we have $[\tilde{y_1}^{i_1}\tilde{y_2}^{i_2}\tilde{w}^{i_3}, (\tilde{y_1}^{i_1}\tilde{y_2}^{i_2})^l\tilde{w}^{j_3}]=e$. Now, from Lemma \ref{commutator_in_center} it follows that $[\tilde{y_1}^{i_1}\tilde{y_2}^{i_2}\tilde{w}^{i_3}, (\tilde{y_1}^{i_1}\tilde{y_2}^{i_2})^l\tilde{w}^{j_3}]=[\tilde{y_1}^{i_1}\tilde{y_2}^{i_2},\tilde{w}^{j_3}][\tilde{w}^{i_3},(\tilde{y_1}^{i_1}\tilde{y_2}^{i_2})^l]$. Again, from Lemma \ref{commutator_in_center}, we have $[\tilde{y_1}^{i_1}\tilde{y_2}^{i_2},\tilde{w}^{j_3}]=[\tilde{y_1},\tilde{w}]^{i_1j_3}[\tilde{y_2},\tilde{w}]^{i_2j_3}=w_1^{i_1j_3}w_2^{i_2j_3}$ and $[\tilde{w}^{i_3},(\tilde{y_1}^{i_1}\tilde{y_2}^{i_2})^l]=[\tilde{w}^{i_3},\tilde{y_1}^{i_1}\tilde{y_2}^{i_2}]^l=[\tilde{w},\tilde{y_1}]^{li_1i_3}[\tilde{w},\tilde{y_2}]^{li_2i_3}=w_1^{-li_1i_3}w_2^{-li_2i_3}$. So we can conclude that $[\tilde{y_1}^{i_1}\tilde{y_2}^{i_2}\tilde{w}^{i_3}, (\tilde{y_1}^{i_1}\tilde{y_2}^{i_2})^l\tilde{w}^{j_3}]=w_1^{i_1j_3-li_1i_3}w_2^{i_2j_3-li_2i_3}=e$ if and only if $w_1^{i_1j_3-li_1i_3}=w_2^{i_2j_3-li_2i_3}=e$, if and only if $i_1j_3-li_1i_3\equiv 0~({\rm mod}~p)$ and $i_2j_3-li_2i_3\equiv 0~({\rm mod}~p)$. Now, since $i_1, i_2\neq 0$, we have $j_3\equiv li_3~({\rm mod}~p)$. So, $w^{j_3}=w^{li_3}$, i.e.~$(y_1^{i_1}y_2^{i_2})^lw^{j_3}\in\langle y_1^{i_1}y_2^{i_2}w^{i_3}\rangle$. Hence, we have $\mathcal{N}_{\mathcal{P}_e(H_3(\mathbb{Z}/p\mathbb{Z}))}[y_1^{i_1}y_2^{i_2}w^{i_3}]= \mathcal{N}_{\Delta_D(H_3(\mathbb{Z}/p\mathbb{Z}))}[y_1^{i_1}y_2^{i_2}w^{i_3}]$.

\vspace{.3cm}
\noindent\underline{\textbf{Case (ii) ($i_1=0, i_2\neq 0$):}} Here, 
$\begin{aligned}[t]
 i_1j_2-i_2j_1&\equiv 0~({\rm mod}~p)\\
 \Rightarrow i_2j_1&\equiv 0~({\rm mod}~p)\\
 \Rightarrow j_1&\equiv 0~({\rm mod}~p).\\
\end{aligned}$\\
Hence, \[C_{H_3(\mathbb{Z}/p\mathbb{Z})}(y_2^{i_2}w^{i_3})=\langle y_2^{i_2}w^{i_3}\rangle\times\langle w\rangle.\] Clearly, $\langle y_2^{i_2}w^{i_3}\rangle=\mathcal{N}_{\mathcal{P}_e(H_3(\mathbb{Z}/p\mathbb{Z}))}[y_2^{i_2}w^{i_3}]\subseteq \mathcal{N}_{\Delta_D(H_3(\mathbb{Z}/p\mathbb{Z}))}[y_2^{i_2}w^{i_3}]$. Let $y_2^{li_2}w^{j_3}$ be an element of $C_{H_3(\mathbb{Z}/p\mathbb{Z})}(y_2^{i_2}w^{i_3})$ such that $y_2^{li_2}w^{j_3}\sim y_2^{i_2}w^{i_3}$ in $\Delta_D(H_3(\mathbb{Z}/p\mathbb{Z}))$. Then we have $[\tilde{y_2}^{i_2}\tilde{w}^{i_3}, \tilde{y_2}^{li_2}\tilde{w}^{j_3}]=e$. From Lemma \ref{commutator_in_center} we have $[\tilde{y_2}^{i_2}\tilde{w}^{i_3}, \tilde{y_2}^{li_2}\tilde{w}^{j_3}]=[\tilde{y_2}^{i_2}, \tilde{w}^{j_3}][\tilde{w}^{i_3}, \tilde{y_2}^{li_2}]=w_2^{i_2j_3-li_2i_3}=e$ if and only if $i_2j_3-li_2i_3\equiv 0~({\rm mod}~p)$. Since $i_2\neq 0$, we have $j_3\equiv li_3~({\rm mod}~p)$. So, $w^{j_3}=w^{li_3}$, i.e.~$y_2^{li_2}w^{j_3}\in \langle y_2^{i_2}w^{i_3}\rangle$. Hence, we have $\mathcal{N}_{\mathcal{P}_e(H_3(\mathbb{Z}/p\mathbb{Z}))}[y_2^{i_2}w^{i_3}]= \mathcal{N}_{\Delta_D(H_3(\mathbb{Z}/p\mathbb{Z}))}[y_2^{i_2}w^{i_3}]$.

\vspace{.3cm}
\noindent\underline{\textbf{Case (iii) ($i_1\neq 0, i_2=0$):}} Here, 
$\begin{aligned}[t]
 i_1j_2-i_2j_1&\equiv 0~({\rm mod}~p)\\
 \Rightarrow i_1j_2&\equiv 0~({\rm mod}~p)\\
 \Rightarrow j_2&\equiv 0~({\rm mod}~p).\\
\end{aligned}$\\
Now, using an argument similar to the above case, it follows that $\mathcal{N}_{\mathcal{P}_e(H_3(\mathbb{Z}/p\mathbb{Z}))}[y_1^{i_1}w^{i_3}]= \mathcal{N}_{\Delta_D(H_3(\mathbb{Z}/p\mathbb{Z}))}[y_1^{i_1}w^{i_3}]$.

\vspace{.3cm}
\noindent\underline{\textbf{Case (iv) ($i_1=i_2=0$):}} Here, \[C_{H_3(\mathbb{Z}/p\mathbb{Z})}(w^{i_3})=H_3(\mathbb{Z}/p\mathbb{Z}).\] Clearly, $\langle w^{i_3}\rangle=\mathcal{N}_{\mathcal{P}_e(H_3(\mathbb{Z}/p\mathbb{Z}))}[w^{i_3}]\subseteq \mathcal{N}_{\Delta_D(H_3(\mathbb{Z}/p\mathbb{Z}))}[w^{i_3}]$. Let $y_1^{j_1}y_2^{j_2}w^{j_3}$ be an element of $C_{H_3(\mathbb{Z}/p\mathbb{Z})}(w^{i_3})$ such that $y_1^{j_1}y_2^{j_2}w^{j_3}\sim w^{i_3}$ in $\Delta_D(H_3(\mathbb{Z}/p\mathbb{Z}))$. Then, $[\tilde{w}^{i_3},\tilde{y_1}^{j_1}\tilde{y_2}^{j_2}\tilde{w}^{j_3}]=e$. Now, $[\tilde{w}^{i_3},\tilde{y_1}^{j_1}\tilde{y_2}^{j_2}\tilde{w}^{j_3}]=[\tilde{w}^{i_3},\tilde{y_1}^{j_1}][\tilde{w}^{i_3},\tilde{y_2}^{j_2}]=w_1^{-i_3j_1}w_2^{-i_3j_2}=e$ if and only if $w_1^{-i_3j_1}=w_2^{-i_3j_2}=e$, if and only if $i_3j_1\equiv i_3j_2\equiv 0~({\rm mod}~p)$. Since $i_3\neq 0$, we have $j_1\equiv j_2\equiv 0~({\rm mod}~p)$. Hence, the result follows.
\end{proof}

Our next results study the dominant vertices of $\Delta_D(H_3(\mathbb{Z}/p^k\mathbb{Z}))$ and connectivity of the reduced graph $\Delta_D(H_3(\mathbb{Z}/p^k\mathbb{Z}))^*$. 
\begin{proposition}\label{dominant_heisenberg}
For the Heisenberg group $H_3(\mathbb{Z}/p^k\mathbb{Z})$, ${_d}\Delta_D(H_3(\mathbb{Z}/p^k\mathbb{Z}))=\{e\}$.    
\end{proposition}
\begin{proof}
Since $\Delta_D(H_3(\mathbb{Z}/p^k\mathbb{Z}))$ is a spanning subgraph of $\Delta(H_3(\mathbb{Z}/p^k\mathbb{Z}))$, it is easy to see that ${_d}\Delta_D(H_3(\mathbb{Z}/p^k\mathbb{Z}))\subseteq{_d}\Delta(H_3(\mathbb{Z}/p^k\mathbb{Z}))$. Now, ${_d}\Delta(H_3(\mathbb{Z}/p^k\mathbb{Z}))=\langle w\rangle$. Hence, any element of ${_d}\Delta_D(H_3(\mathbb{Z}/p^k\mathbb{Z}))$ must be of the form $w^\alpha$, where $0\leq \alpha<p^k$. Let $w^\alpha\in{_d}\Delta_D(H_3(\mathbb{Z}/p^k\mathbb{Z}))$, with $0\leq \alpha<p^k$. Then $w^\alpha\sim y_1$ in $\Delta_D(H_3(\mathbb{Z}/p^k\mathbb{Z}))$, i.e.~$[\tilde{w}^\alpha,\tilde{y_1}]=e$. Now, from Lemma \ref{commutator_in_center} and equation (\ref{heisenberg_schur_cover}) we have $[\tilde{w}^\alpha,\tilde{y_1}]=[\tilde{w},\tilde{y_1}]^\alpha=e$ if and only if $\alpha=0$. Hence, the result follows. 
\end{proof}
\begin{proposition}\label{connectivity_heisenberg}
For $k\in\mathbb{N}$, $\Delta_D(H_3(\mathbb{Z}/p^k\mathbb{Z}))^*$ is disconnected if and only if $k=1$.  
\end{proposition}
\begin{proof}
For $k=1$, Theorem \ref{heisenberg_equality} concludes that $\Delta_D(H_3(\mathbb{Z}/p^k\mathbb{Z}))=\mathcal{P}_e(H_3(\mathbb{Z}/p^k\mathbb{Z}))$. So $\Delta_D(H_3(\mathbb{Z}/p^k\mathbb{Z}))^*=\mathcal{P}_e(H_3(\mathbb{Z}/p^k\mathbb{Z}))^*$. It was shown in \cite[Theorem 28]{aa} that for any finite $p$-group $G$, $\mathcal{P}_e(G)=\mathcal{P}(G)$  and in \cite[Corollary 4.1]{moga} that $\mathcal{P}(G)\setminus\{e\}$ is disconnected if $G$ is neither cyclic nor isomorphic to $Q_{2^n}$. Hence $\Delta_D(H_3(\mathbb{Z}/p^k\mathbb{Z}))^*$ is disconnected.

For $k\geq 2$, let $y_1^{i_1}y_2^{i_2}w^{i_3}\in V(H_3(\mathbb{Z}/p^k\mathbb{Z})^*)$. If $p$ divides both $i_1$ and $i_2$, then from Lemma \ref{commutator_in_center} and equation (\ref{heisenberg_schur_cover}) we have $[\tilde{y_1}^{i_1}\tilde{y_2}^{i_2}\tilde{w}^{i_3},\tilde{w}^{p^{k-1}}]=[\tilde{y_1},\tilde{w}]^{i_1p^{k-1}}[\tilde{y_2},\tilde{w}]^{i_2p^{k-1}}=e$. So $y_1^{i_1}y_2^{i_2}w^{i_3}\sim w^{p^{k-1}}$. Suppose that at least one of $i_1$ and $i_2$ is not divisible by $p$. Without loss of generality, we assume that $p\nmid i_1$. Then $(y_1^{i_1}y_2^{i_2}w^{i_3})^p=y_1^{i_1p}y_2^{i_2p}w^{(i_3p-i_1i_2\frac{p(p-1)}{2})}\neq e$, since $p^k\nmid i_ip$. Now, $[\tilde{y_1}^{i_1p}\tilde{y_2}^{i_2p}\tilde{w}^{({i_3p-i_1i_2\frac{p(p-1)}{2}})},\tilde{w}^{p^{k-1}}]=[\tilde{y_1},\tilde{w}]^{i_1p^k}[\tilde{y_2},\tilde{w}]^{i_2p^k}=e$. So, $y_1^{i_1}y_2^{i_2}w^{i_3}\sim (y_1^{i_1}y_2^{i_2}w^{i_3})^p\sim w^{p^{k-1}}$ is a path from $y_1^{i_1}y_2^{i_2}w^{i_3}$ to $w^{p^{k-1}}$. So $\Delta(H_3(\mathbb{Z}/p^k\mathbb{Z}))^*$ is connected. Hence, the result follows.
\end{proof}

\begin{theorem}\label{diameter_heisenberg}
For any positive integer $k\geq 2$, ${\rm diam}(\Delta_D(H_3(\mathbb{Z}/p^k\mathbb{Z}))^*)\leq 4$.    
\end{theorem}

\subsection{Symmetric and alternating groups}\label{Deep commuting graph of symmetric and alternating groups}

Throughout this subsection, for any permutation $\tau$, by $\tilde{\tau}$ we denote a preimage of $\tau$ in the corresponding Schur cover, unless otherwise mentioned. In order to understand the structure of the deep commuting graph of $S_n$ and $A_n$, let us first look at the Schur multiplier and Schur cover of $S_n$ and $A_n$. 

For $n=3$, both $M(S_3)$ and $M(A_3)$ are trivial. For $n\geq 4$, the Schur multiplier and Schur covers of $S_n$ and $A_n$ are described in the following results. 
\begin{theorem}[{\rm\cite[Theorem 2.12.3]{karpi}}]\label{symmetric_schur_cover}
 For $n\geq 4$, $M(S_n)=C_2$. Define the group $\tilde{S_n}$ by 
 \[\tilde{S_n}=\langle g_1,g_2,\dots,g_{n-1},z~\vert~ g_i^2=(g_jg_{j+1})^3=(g_kg_l)^2=z,~z^2=[z,g_i]=e\rangle,\]
 where $1\leq i\leq n-1,~1\leq j\leq n-2,~k\leq l-2$. Then $\tilde{S_n}$ is a Schur cover of $S_n$.
\end{theorem}
\noindent Also, we have $\pi(g_i)=t_i$ and $\pi(z)=e$, where $t_i$ is the transposition $(i,i+1)$ in $S_n$ and $\pi$ is the projection map from $\tilde{S_n}$ onto $S_n$.
\begin{theorem}[{\cite[Theorem 2.12.5]{karpi}}]\label{alternating_schur_cover}
Let $n\geq 4$. Then $M(A_n)=C_2$ when $n\notin \{6,7\}$, and $M(A_n)=C_6$ when $n\in \{6,7\}$. Define the group $\tilde{A_n}$ as follows:
\begin{enumerate}[\rm (i)]
    \item $\tilde{A_n}=[\tilde{S_n},\tilde{S_n}]$ if $n\notin \{6,7\}$;
    \item $\begin{aligned}\\
     \tilde{A_n}=\langle h_1,\dots,h_{n-2},\mathfrak{z}~\vert~h_1^3=h_i^2=(h_{i-1}h_i)^3=&(h_jh_k)^2=\mathfrak{z}^3,\\
     &(h_1h_4)^2=\mathfrak{z},~\mathfrak{z}^6=[\mathfrak{z},h_t]=e\rangle, 
    \end{aligned}$\\
    where $2\leq i\leq n-2,~1\leq j<k-1,~k\leq n-2,~(j,k)\neq(1,4)$ and $1\leq t\leq n-2$, for $n\in \{6,7\}$.
\end{enumerate}
Then $\tilde{A_n}$ is the Schur cover of $A_n$.
\end{theorem}
\noindent Also, we have $\pi(h_i)=q_i$ and $\pi(\mathfrak{z})=e$, where $q_1$ is the $3$-cycle $(123)$ and $q_i$ is the permutation $(12)(i+1,i+2)$ for $2\leq i\leq n-2$, and $\pi$ is the projection map from $\tilde{A_n}$ onto $A_n$.

It is easy to see, that for any permutation $s\in S_n$, if $\tilde{s}$ is a preimage of $s$ in $\tilde{S_n}$, then $\pi^{-1}(s)=\{\tilde{s},z\tilde{s}\}$. With the above observation, we have the following lemma.
\begin{lemma}\label{preimage_transposition}
Let $\tau_1$ and $\tau_2$ be two disjoint transpositions in $S_n$ and $\tilde{\tau_1}$ and $\tilde{\tau_2}$ be preimages of $\tau_1$ and $\tau_2$ in $\tilde{S_n}$, respectively. Then $(\tilde{\tau_1})^2=(\tilde{\tau_2})^2=(\tilde{\tau_1}\tilde{\tau_2})^2=[\tilde{\tau_1},\tilde{\tau_2}]=z$.     
\end{lemma}
\begin{proof}
Let $\tau_1=(ij)$ and $\tau_2=(kl)$ such that $i,j,k,l$ are four distinct elements from $[n]$. Let $\sigma\in S_n$ such that $\sigma(i)=1,~\sigma(j)=2,~\sigma(k)=3,~\sigma(l)=4$. Then $(12)^{\sigma}=\tau_1$. Let $\tilde{\sigma}$ be a preimage of $\sigma$ in $\tilde{S_n}$. Then $\tilde{\tau_1}=g_1^{\tilde{\sigma}}$ or $\tilde{\tau_1}=zg_1^{\tilde{\sigma}}$ . Hence, $(\tilde{\tau_1})^2=(g_1^{\tilde{\sigma}})^2=(zg_1^{\tilde{\sigma}})^2=z^{\tilde{\sigma}}=z.$ Similarly, it can be shown that $(\tilde{\tau_2})^2=z$ and $(\tilde{\tau_1}\tilde{\tau_2})^2=z$. Also, $[\tilde{\tau_1},\tilde{\tau_2}]=\tilde{\tau_1}^{-1}\tilde{\tau_2}^{-1}\tilde{\tau_1}\tilde{\tau_2}=(\tilde{\tau_1}\tilde{\tau_2})^2=z$. Hence, the result follows.
\end{proof}
Again, it is easy to see that if $s_1,s_2\in S_n$ (or $A_n$) such that $[s_1,s_2]=e$, then $[\tilde{s_1},\tilde{s_2}]\in Z(\tilde{S_n})$ (or $Z(\tilde{A_n})$) where $\tilde{s_1},\tilde{s_2}$ are preimages of $s_1,s_2$ in $\tilde{S_n}$ (or $\tilde{A_n}$), respectively. From the above observation and Lemma \ref{commutator_in_center}, we have the following.

\begin{lemma}\label{disjoint_commutators}
Let $s_1,s_2,s_3\in S_n$ (or $A_n$) with $\tilde{s_1},\tilde{s_2},\tilde{s_3}$ being their preimages in $\tilde{S_n}$ (or $\tilde{A_n}$), respectively. Then the following happens:
\begin{enumerate}[\rm (i)]
\item If $[s_1,s_3]=e$, then $[\tilde{s_1}\tilde{s_2},\tilde{s_3}]=[\tilde{s_1},\tilde{s_3}][\tilde{s_2},\tilde{s_3}]$;
\item If $[s_1,s_2]=e$, then $[\tilde{s_1},\tilde{s_2}\tilde{s_3}]=[\tilde{s_1},\tilde{s_2}][\tilde{s_1},\tilde{s_3}]$.
\end{enumerate}    
\end{lemma}

\begin{proposition}\label{induced_over_alternating}
For $n\geq 4$ and $n\notin \{6,7\}$, we have $\Delta_D(S_n)\vert_{A_n}=\Delta_D(A_n)$.    
\end{proposition}
\begin{proof}
For $n\geq 4$ and $n\notin \{6,7\}$, we have $M(A_n)\cong M(S_n)\cong C_2$. Let $\pi$ denote the projection map from $\tilde{S_n}$ onto $S_n$. Let $\tilde{\tau_1},\tilde{\tau_2},\tilde{\tau_3},\tilde{\tau_4}$ be preimages of $(12),(34),(13),(24)$ in $\tilde{S_n}$, respectively. Then $[\tilde{\tau_1}\tilde{\tau_2},\tilde{\tau_3}\tilde{\tau_4}]\in [\pi^{-1}(A_n),\pi^{-1}(A_n)]$. Now, from Lemma \ref{preimage_transposition} we have $[\tilde{\tau_1}\tilde{\tau_2},\tilde{\tau_3}\tilde{\tau_4}]=(\tilde{\tau_1}\tilde{\tau_2})^{-1}(\tilde{\tau_3}\tilde{\tau_4})^{-1}\tilde{\tau_1}\tilde{\tau_2}\tilde{\tau_3}\tilde{\tau_4}=(\tilde{\tau_1}\tilde{\tau_2}\tilde{\tau_3}\tilde{\tau_4})^2$. Also, $\pi(\tilde{\tau_1}\tilde{\tau_2}\tilde{\tau_3}\tilde{\tau_4})=(14)(23)$. Let $\tilde{\tau_5}$ and $\tilde{\tau_6}$ be preimages of $(14)$ and $(23)$, respectively. Then $\tilde{\tau_1}\tilde{\tau_2}\tilde{\tau_3}\tilde{\tau_4}=\tilde{\tau_5}\tilde{\tau_6}$, or $\tilde{\tau_1}\tilde{\tau_2}\tilde{\tau_3}\tilde{\tau_4}=z\tilde{\tau_5}\tilde{\tau_6}$. So, we have $[\tilde{\tau_1}\tilde{\tau_2},\tilde{\tau_3}\tilde{\tau_4}]=(\tilde{\tau_1}\tilde{\tau_2}\tilde{\tau_3}\tilde{\tau_4})^2=(\tilde{\tau_5}\tilde{\tau_6})^2=(z\tilde{\tau_5}\tilde{\tau_6})^2=z$. Hence, $\iota(M(S_n))=\{z,e\}\subseteq [\pi^{-1}(A_n),\pi^{-1}(A_n)]$. Now, from Proposition \ref{induced_subgraph_equality} the result follows.
\end{proof}
\begin{proposition}\label{induced_S_n and A_n}
 For $4\leq m< n$, consider the subgroup $S_m$ of $S_n$ being the group of permutations of the first $m$ symbols inside $S_n$ and the subgroup $A_m$ of $A_n$ being the group of even permutations of the first $m$ symbols inside $A_n$. Then the following happens:
 \begin{enumerate}[\rm (i)]
     \item $\Delta_D(S_n)\vert_{S_m}=\Delta_D(S_m)$;
     \item For $m,n\notin\{6,7\}$, we have $\Delta_D(A_n)\vert_{A_m}=\Delta_D(A_m)$;
     \item $\Delta_D(A_7)\vert_{A_6}=\Delta_D(A_6)$.
 \end{enumerate}    
\end{proposition}
\begin{proof}
\noindent
\begin{enumerate}[\rm (i)]
     \item The subgroup $S_m$ of $S_n$ is considered to be the group of permutations of the first $m$ symbols inside $S_n$. Clearly, $M(S_m)\cong M(S_n)\cong C_2$. Let $\pi$ denote the projection map from $\tilde{S_n}$ onto $S_n$. Since $m\geq 4$, consider the permutations $(12),(34)\in S_m\subseteq S_n$. Let $\tilde{\tau_1},\tilde{\tau_2}$ be preimages of $(12),(34)$ in $\tilde{S_n}$, respectively. Clearly $[\tilde{\tau_1},\tilde{\tau_2}]\in[\pi^{-1}(S_m),\pi^{-1}(S_m)]$. Again, from Lemma \ref{preimage_transposition} we have $[\tilde{\tau_1},\tilde{\tau_2}]=z.$ Hence, $\iota(M(S_n))=\{z,e\}\subseteq [\pi^{-1}(S_m),\pi^{-1}(S_m)]$. Now, the result follows from Proposition \ref{induced_subgraph_equality}.  
     
     \item The subgroup $A_m$ of $A_n$ is considered to be the group of even permutations of the first $m$ symbols inside $A_n$. Since $m,n\notin\{6,7\}$, we have $M(A_m)\cong M(A_n)\cong C_2$. Let $\pi_{_A}$ denote the projection map from $\tilde{A_n}$ onto $A_n$. Since $m\geq 4$, consider the permutations $(12)(34),(13)(24)\in A_m\subseteq A_n$. Let $\tilde{(12)(34)},\tilde{(13)(24)}$ denote respective preimages of $(12)(34),(13)(24)$ in $\tilde{A_n}$. Then $[\tilde{(12)(34)},\tilde{(13)(24)}]\in[\pi_{_A}^{-1}(A_m),\pi_{_A}^{-1}(A_m)]\subseteq\tilde{A_n}$. Since, for $m,n\notin\{6,7\}$, $\tilde{A_n}=[\tilde{S_n},\tilde{S_n}]\subseteq\tilde{S_n}$ and $\pi_{_A}:\tilde{A_n}\rightarrow A_n$ is the same as the restriction of $\pi:\tilde{S_n}\rightarrow S_n$ over $\tilde{A_n}$, we have $[\tilde{(12)(34)},\tilde{(13)(24)}]=[\tilde{\tau_1}\tilde{\tau_2},\tilde{\tau_3}\tilde{\tau_4}]$, where $\tilde{\tau_1},\tilde{\tau_2},\tilde{\tau_3},\tilde{\tau_4}$ are respective preimages of $(12),(34),(13),(24)$ in $S_n$. Now, from the proof of Proposition \ref{induced_over_alternating}, we have $[\tilde{\tau_1}\tilde{\tau_2},\tilde{\tau_3}\tilde{\tau_4}]=z.$ Hence, $\iota(M(A_n))=\{z,e\}\subseteq [\pi_{_A}^{-1}(A_m),\pi_{_A}^{-1}(A_m)]$. Now, the rest follows from Proposition \ref{induced_subgraph_equality}.

     \item The subgroup $A_6$ of $A_7$ is considered to be the group of even permutations of the first $6$ symbols inside $A_7$. Let $\pi_{_{A_7}}$ denote the projection map from $\tilde{A_7}$ onto $A_7$. We have $M(A_6)\cong M(A_7)\cong C_6$. Consider $(123),(12)(34),(12)(45),(12)(56)\in A_6\subseteq A_7$. Then from Theorem \ref{alternating_schur_cover}(ii) we can say that $h_1,h_2,h_3,h_4$ are preimages of $(123),(12)(34),(12)(45),(12)(56)$ in $A_7$, respectively. Now $[h_1,h_2],[h_1,h_3],[h_1,h_4]\in[\pi_{_{A_7}}^{-1}(A_6),\pi_{_{A_7}}^{-1}(A_6)]$. Also we have, $[h_1,h_2]=h_1^{-1}h_2^{-1}h_1h_2=h_1^2h_2h_1h_2=\mathfrak{z}^3h_1(h_1h_2)^{-1}$ and $[h_1,h_3]=h_1^{-1}h_3^{-1}h_1h_3=h_1^2h_3h_1h_3=\mathfrak{z}^3h_1$. So, both $\mathfrak{z}^3h_1$ and $h_1h_2$ are inside $[\pi_{_{A_7}}^{-1}(A_6),\pi_{_{A_7}}^{-1}(A_6)].$ In addition, $(h_1h_2)^3=\mathfrak{z}^3\in[\pi_{_{A_7}}^{-1}(A_6),\pi_{_{A_7}}^{-1}(A_6)]$, so $h_1\in[\pi_{_{A_7}}^{-1}(A_6),\pi_{_{A_7}}^{-1}(A_6)]$. Now, $h_1^{-1}[h_1,h_4]=h_1^{-2}h_4^{-1}h_1h_4=(h_1h_4)^2=\mathfrak{z}$. Hence $\iota(M(A_7))=\langle\mathfrak{z}\rangle\subseteq[\pi_{_{A_7}}^{-1}(A_6),\pi_{_{A_7}}^{-1}(A_6)]$. Now, the result follows from Proposition \ref{induced_subgraph_equality}. 
\end{enumerate}
\end{proof}

\begin{lemma}\label{disjoint_permutations_in_Sn}
Let $n\geq 4$ and $\sigma_1$, $\sigma_2$ be two disjoint permutations in $S_n$. Then $\sigma_1\sim\sigma_2$ in $\Delta_D(S_n)$ if and only if at least one of $\sigma_1$ and $\sigma_2$ is an even permutation.  
\end{lemma}
\begin{proof}
Let $\sigma_1$ and $\sigma_2$ be two disjoint permutations in $S_n$. Let $\sigma_1=\tau_{11}\tau_{12}\dots\tau_{1k}$ and $\sigma_2=\tau_{21}\tau_{22}\dots\tau_{2m}$, where $\tau_{ij}$'s are transpositions such that $\tau_{1i}$ and $\tau_{2j}$ are disjoint, for all $i\in [k]$ and $j\in [m]$. Then $\tilde{\tau_{11}}\tilde{\tau_{12}}\dots\tilde{\tau_{1k}}$ and $\tilde{\tau_{21}}\tilde{\tau_{22}}\dots\tilde{\tau_{2m}}$ are preimages of $\sigma_1$ and $\sigma_2$ in $\tilde{S_n}$, respectively. So, using Lemma \ref{preimage_transposition} and \ref{disjoint_commutators} we have $[\tilde{\sigma_1},\tilde{\sigma_2}]=[\tilde{\tau_{11}}\tilde{\tau_{12}}\dots\tilde{\tau_{1k}},\tilde{\tau_{21}}\tilde{\tau_{22}}\dots\tilde{\tau_{2m}}]=z^{km}$, and $z^{km}=e$ if and only if at least one of $k$ and $m$ is even. This concludes the proof. 
\end{proof}
As a consequence of the above result and Proposition \ref{induced_over_alternating} we arrive at the following result about $\Delta_D(A_n)$.
\begin{corollary}\label{disjoint_permutations_in_An}
Let $n\geq 8$ and $\sigma_1$, $\sigma_2$ be two disjoint permutations in $A_n$. Then $\sigma_1\sim\sigma_2$ in $\Delta_D(A_n)$.    
\end{corollary}

It is easy to see that for any group $G$, $\mathcal{N}_{\Delta(G)}[x]=C_{G}(x)$, for all $x\in G$. Then the next observation follows from Theorem \ref{inclusion}.
\begin{observation}\label{deep commuting implies inside centralizer}
Let $n\in \mathbb{N}$ and $x\in S_n$ (or $A_n$). Then $\mathcal{N}_{\mathcal{P}_e(S_n)}(x)\subseteq \mathcal{N}_{\Delta_D(S_n)}(x)\subseteq C_{S_n}(x)$ (or $\mathcal{N}_{\mathcal{P}_e(A_n)}(x)\subseteq \mathcal{N}_{\Delta_D(A_n)}(x)\subseteq C_{A_n}(x)$).    
\end{observation}

\begin{proposition}\label{deep commuting of S_4,5}
For $n\in\{4,5\}$, we have $\Delta_D(S_n)=\mathcal{P}_e(S_n)$.
\end{proposition}

\begin{proof}
We first prove the case for $n=5$. We claim that for any permutation $\tau\in S_5$, $\mathcal{N}_{\Delta_D(S_5)}(\tau)=\mathcal{N}_{\mathcal{P}_e(S_5)}(\tau)$, which yields $\Delta_D(S_5)=\mathcal{P}_e(S_5)$. To show the above claim, let us consider the following cases:  

\vspace{.3cm}
\noindent\underline{\textbf{Case (i) ($\tau=(a,b)$):}} Consider the permutation $(12)$. Then \[C_{S_5}((12))=\langle (12)\rangle\times\langle (345),(34)\rangle.\] Clearly, $\{e,(345),(354),(12)(345),(12)(354)\}=\mathcal{N}_{\mathcal{P}_e(S_5)}((12))\subseteq \mathcal{N}_{\Delta_D(S_5)}((12))$. Now, from Lemma \ref{preimage_transposition} we have $[\tilde{(12)},\tilde{(34)}]=[\tilde{(12)},\tilde{(35)}]=[\tilde{(12)},\tilde{(45)}]=z$. Also, from Lemma \ref{preimage_transposition} and \ref{disjoint_commutators} we have $[\tilde{(12)},\tilde{(12)}\tilde{(34)}]=[\tilde{(12)},\tilde{(12)}\tilde{(35)}]=[\tilde{(12)},\tilde{(12)}\tilde{(45)}]=z$. Hence, $\mathcal{N}_{\Delta_D(S_5)}((12))=\mathcal{N}_{\mathcal{P}_e(S_5)}((12))$. Using similar arguments, it can be shown that $\mathcal{N}_{\Delta_D(S_5)}((a,b))=\mathcal{N}_{\mathcal{P}_e(S_5)}((a,b))$, for any two-cycle $(a,b)$ in $S_5$.


\vspace{.3cm}
\noindent\underline{\textbf{Case (ii) ($\tau=(a,b)(c,d)$):}} Consider the permutation $(12)(34)$. Then \[C_{S_5}((12)(34))=\{e,(12),(34),(12)(34),(13)(24),(14)(23),(1324),(1423)\}.\]
Clearly, $\{e,(1324),(1423)\}=\mathcal{N}_{\mathcal{P}_e(S_5)}((12)(34))\subseteq \mathcal{N}_{\Delta_D(S_5)}((12)(34))$. Again, from Lemma \ref{preimage_transposition} and \ref{disjoint_commutators}, we have $[\tilde{(12)}\tilde{(34)},\tilde{(12)}]=[\tilde{(34)},\tilde{(12)}]=z$ and $[\tilde{(12)}\tilde{(34)},\tilde{(34)}]=[\tilde{(12)},\tilde{(34)}]=z$. Also, since $(\tilde{(12)}\tilde{(34)})^2=z=(\tilde{(13)}\tilde{(24)})^2$, we have $[\tilde{(12)}\tilde{(34)},\tilde{(13)}\tilde{(24)}]=(\tilde{(12)}\tilde{(34)}\tilde{(13)}\tilde{(24)})^2.$ Now, consider the permutation $(12)(34)(13)(24)=(14)(23)$. Then we have $\tilde{(12)}\tilde{(34)}\tilde{(13)}\tilde{(24)}\in\pi^{-1}((14)(23))$, where $\pi$ is the projection map from $\tilde{S_5}$ to $S_5$. Hence $\tilde{(12)}\tilde{(34)}\tilde{(13)}\tilde{(24)}=\tilde{(14)}\tilde{(23)}$ or $\tilde{(12)}\tilde{(34)}\tilde{(13)}\tilde{(24)}=z\tilde{(14)}\tilde{(23)}$, which implies that $(\tilde{(12)}\tilde{(34)}\tilde{(13)}\tilde{(24)})^2=(\tilde{(14)}\tilde{(23)})^2=(z\tilde{(14)}\tilde{(23)})^2=z.$ Again, it follows that $[\tilde{(12)}\tilde{(34)},\tilde{(14)}\tilde{(23)}]=[\tilde{(12)}\tilde{(34)},\tilde{(12)}\tilde{(34)}\tilde{(13)}\tilde{(24)}]=[\tilde{(12)}\tilde{(34)},\tilde{(12)}\tilde{(34)}][\tilde{(12)}\tilde{(34)},\tilde{(13)}\tilde{(24)}]=(\tilde{(12)}\tilde{(34)}\tilde{(13)}\tilde{(24)})^2=z.$ So, it can be concluded that $\mathcal{N}_{\Delta_D(S_5)}((12)(34))=\mathcal{N}_{\mathcal{P}_e(S_5)}((12)(34)).$ Similarly, for any permutation of the form $(a,b)(c,d)\in S_5$, it can be shown that $\mathcal{N}_{\Delta_D(S_5)}((a,b)(c,d))=\mathcal{N}_{\mathcal{P}_e(S_5)}((a,b)(c,d))$. 

\vspace{.3cm}
\noindent\underline{\bf{Case (iii) ($\tau=(a,b,c)$):}} Consider the permutation $(123)$. Then \[C_{S_5}((123))=\{e,(123),(132),(45),(123)(45),(132)(45)\}=\langle (123)(45)\rangle.\] Here, $\mathcal{N}_{\mathcal{P}_e(S_5)}[(123)]=C_{S_5}((123))$. Hence, we have $\mathcal{N}_{\Delta_D(S_5)}((123))=\mathcal{N}_{\mathcal{P}_e(S_5)}((123))$. Using similar arguments, it can be concluded that for any arbitrary permutation of the form $(a,b,c)\in S_5$, we have $\mathcal{N}_{\Delta_D(S_5)}((a,b,c))=\mathcal{N}_{\mathcal{P}_e(S_5)}((a,b,c))$.

\vspace{.3cm}
\noindent\underline{\bf{Case (iv) ($\tau=(a,b,c)(d,e)$):}} Consider the permutation $(123)(45)$. Then \[C_{S_5}((123)(45))=\{e,(123),(132),(45),(123)(45),(132)(45)\}=\langle (123)(45)\rangle.\] Here, $\mathcal{N}_{\mathcal{P}_e(S_5)}[(123)(45)]=C_{S_5}((123)(45))$. So, $\mathcal{N}_{\Delta_D(S_5)}((123)(45))=\mathcal{N}_{\mathcal{P}_e(S_5)}((123)(45))$. Using similar arguments, it can be concluded that for any arbitrary permutation of the form $(a,b,c)(d,e)\in S_5$, we have $\mathcal{N}_{\Delta_D(S_5)}((a,b,c)(d,e))=\mathcal{N}_{\mathcal{P}_e(S_5)}((a,b,c)(d,e))$.

\vspace{.3cm}
\noindent\underline{\bf{Case (v) ($\tau=(a,b,c,d)$):}} Consider the permutation $(1234)$. Then \[C_{S_5}((1234))=\langle (1234)\rangle.\] So, $\mathcal{N}_{\mathcal{P}_e(S_5)}[(1234)]=C_{S_5}((1234))$. Hence, we have $\mathcal{N}_{\Delta_D(S_5)}((1234))=\mathcal{N}_{\mathcal{P}_e(S_5)}((1234))$. Using similar arguments, it can be concluded that for any arbitrary permutation of the form $(a,b,c,d)\in S_5$, we have $\mathcal{N}_{\Delta_D(S_5)}((a,b,c,d))=\mathcal{N}_{\mathcal{P}_e(S_5)}((a,b,c,d))$.

\vspace{.3cm}
\noindent\underline{\bf{Case (vi) ($\tau=(a,b,c,d,e)$):}} Consider the permutation $(12345)$. Then \[C_{S_5}((12345))=\langle (12345)\rangle.\] So, $\mathcal{N}_{\mathcal{P}_e(S_5)}[(12345)]=C_{S_5}((12345))$. Hence, we have $\mathcal{N}_{\Delta_D(S_5)}((12345))=\mathcal{N}_{\mathcal{P}_e(S_5)}((12345))$. Using similar arguments, it can be concluded that for any arbitrary permutation of the form $(a,b,c,d,e)\in S_5$, we have $\mathcal{N}_{\Delta_D(S_5)}((a,b,c,d,e))=\mathcal{N}_{\mathcal{P}_e(S_5)}((a,b,c,d,e))$.

Hence, we have $\Delta_D(S_5)=\mathcal{P}_e(S_5)$. Now, from Proposition \ref{induced_S_n and A_n}(i) and the fact that the induced subgraph $\mathcal{P}_e(G)\vert_S$, for any subgroup $S$ of $G$ coincides with $\mathcal{P}_e(S)$, the equality can be deduced for $S_4$. This concludes the proof. 
\end{proof}

\begin{theorem}\label{Containment of Delta_D for S_n}
For $n\in \mathbb{N}$, we have the following:
\begin{enumerate}[\rm (i)]
\item $E(\mathcal{P}_e(S_n))=E(\Delta_D(S_n))\subsetneq E(\Delta(S_n))$ when $n\in\{4,5\}$;
\item $E(\mathcal{P}_e(S_n))\subsetneq E(\Delta_D(S_n))\subsetneq E(\Delta(S_n))$ when $n\geq 6$.
\end{enumerate}
\end{theorem}
\begin{proof}
Recall that for $n\geq 4$, $g_1$, $g_3$ are preimages of $t_1=(12)$ and $t_3=(34)$ in $\tilde{S_n}$, respectively. Now, $[g_1,g_3]=(g_1g_3)^2=z$ in $\tilde{S_n}$ but $[t_1,t_3]=e$ in $S_n$. So $t_1\sim t_3$ in $\Delta(S_n)$ but not in $\Delta_D(S_n)$. Hence, for $n\geq 4$, $E(\Delta_D(S_n))\subsetneq E(\Delta(S_n))$.

For $n\in\{4,5\}$, it was shown in Proposition \ref{deep commuting of S_4,5} that $\mathcal{P}_e(S_n)=\Delta_D(S_n)$. Now, for $n\geq 6$, the subgroup $\langle(123),(456)\rangle$ of $S_n$ is isomorphic to $C_3\times C_3$. So, $(123)\nsim (456)$ in $\mathcal{P}_e(S_n)$. Again, from Lemma \ref{disjoint_permutations_in_Sn} we have $(123)\sim (456)$ in $\Delta_D(S_n)$. Hence, for $n\geq 6$, $E(\mathcal{P}_e(S_n))\subsetneq E(\Delta_D(S_n))$. This completes the proof. 
\end{proof}

Next, we compare the deep commuting graph and the enhanced power graph of $A_n$. We need the following useful lemmas.
\begin{lemma}\label{3-cycle A_7}
Let $\sigma_1$ and $\sigma_2$ be two disjoint $3$-cycles in $A_7$. Then $\sigma_1\nsim\sigma_2$ in $\Delta_D(A_7)$.    
\end{lemma}
\begin{proof}
Let $\sigma_1=(a,b,c)$ and $\sigma_2=(d,e,f)$ be two disjoint $3$-cycles in $A_7$. Consider the $3$-cycles $(123)$ and $(456)$ in $A_7$. Then $[\tilde{(123)},\tilde{(456)}]$=$[h_1,h_3h_4]=[h_1,h_4][h_1,h_3]^{h_4}.$ Also, $[h_1,h_4]=h_1^{-1}h_4^{-1}h_1h_4=h_1(h_1h_4)^2=\mathfrak{z}h_1$, and $[h_1,h_3]^{h_4}=(h_1^{-1}h_3^{-1}h_1h_3)^{h_4}=(h_1(h_1h_3)^2)^{h_4}=\mathfrak{z}^3h_1^{h_4}=h_4h_1h_4=\mathfrak{z}h_1^{-1}.$ Then $[\tilde{(123)},\tilde{(456)}]=\mathfrak{z}h_1\cdot \mathfrak{z}h_1^{-1}=\mathfrak{z}^2$. Hence, $(123)\nsim (456)$ in $\Delta_D(A_7)$.

Again, since $\sigma_1$ and $\sigma_2$ are disjoint, there exists an automorphism $f:A_7\rightarrow A_7$, such that $f((123))=\sigma_1$ and $f((456))=\sigma_2$. Now, from the fact that the deep commuting graph of a finite group is invariant under any automorphism of the group, the result follows. 
\end{proof}

\begin{lemma}\label{2-cycle A_7}
Let $s_1$ and $s_2$ be two distinct permutations of order $2$ in $A_7$ such that $[s_1,s_2]=e$. Then $s_1\nsim s_2$ in $\Delta_D(A_7)$.    
\end{lemma}
\begin{proof}
Since, $s_1\in A_7$ is of order $2$, $s_1$ must be of the form $(a,b)(c,d)$, where $(a,b)$ and $(c,d)$ are disjoint. Also, since $[s_1,s_2]=e=s_2^2$ and $s_1\neq s_2$, $s_2$ must be one of the form $(a,c)(b,d)$, $(a,d)(b,c)$, $(a,b)(e,f)$ and $(c,d)(e,f)$, where $\{a,b,c,d\}\cap\{e,f\}=\emptyset$. Now, consider the permutations $(12)(34)$ and $(13)(24)$. Then we have $[\tilde{(12)(34)},\tilde{(13)(24)}]=[h_2,h_2^{h_1^2h_4}]=(h_2h_2^{h_1^2h_4})^2$. Also, $h_2h_2^{h_1^2h_4}=h_2h_4^{-1}h_1^{-2}h_2h_1^2h_4=h_2h_4h_1h_2h_1^2h_4=\mathfrak{z}^3h_2h_4(h_1h_2)^2h_2h_1h_4=h_2h_4(h_1h_2)^{-1}h_2h_1h_4=h_2h_4h_2^{-1}h_1^{-1}h_2h_1h_4=\mathfrak{z}^3(h_2h_4)^2h_4^{-1}h_1^{-1}h_2h_1h_4=h_2^{h_1h_4}.$ So, we have $(h_2h_2^{h_1^2h_4})^2=(h_2^2)^{h_1h_4}=\mathfrak{z}^3$. Hence, $(12)(34)\nsim (13)(24)$ in $\Delta_D(A_7)$. Again $(14)(23)=(12)(34)(13)(24)$. Then it follows that $[\tilde{(12)(34)},\tilde{(14)(23)}]=[\tilde{(12)(34)},\tilde{(12)(34)}\tilde{(13)(24)}]=[\tilde{(12)(34)},\tilde{(13)(24)}]=\mathfrak{z}^3$. Hence, $(12)(34)\nsim (14)(23)$ in $\Delta_D(A_7)$. Moreover, consider the permutation $(12)(56)$. Then we have, $[\tilde{(12)(34)},\tilde{(12)(56)}]=[h_2,h_4]=(h_2h_4)^2=\mathfrak{z}^3$. Hence, $(12)(34)\nsim (12)(56)$ in $\Delta_D(A_7)$.

Again, there exists an automorphism $f_1:A_7 \rightarrow A_7$ such that, $f_1((12)(34))=(a,b)(c,d)$, $f_1((13)(24))=(a,c)(b,d)$ and $f_1((14)(23))=(a,d)(b,c)$. Also, there exist automorphisms $f_2,f_3:A_7 \rightarrow A_7$ with, $f_2((12)(34))=f_3((12)(34))=(a,b)(c,d)$, $f_2((12)(56))=(a,b)(e,f)$ and $f_3((12)(56))=(c,d)(e,f)$. Now, from the fact that the deep commuting graph of a finite group is invariant under any automorphism of the group, the result follows.
\end{proof}
\begin{proposition}\label{deep commuting of A_6,7}
For $n\in\{6,7\}$, we have $\Delta_D(A_n)=\mathcal{P}_e(A_n)$.    
\end{proposition}

\begin{proof}
The idea of this proof is similar to the proof of Proposition \ref{deep commuting of S_4,5}. We claim that for any permutation $\sigma\in A_7$, $\mathcal{N}_{\Delta_D(A_7)}(\sigma)=\mathcal{N}_{\mathcal{P}_e(A_7)}(\sigma)$, which yields $\Delta_D(A_7)=\mathcal{P}_e(A_7)$. To show the above claim, we consider the following cases:  

\vspace{.3cm}
\noindent\underline{\bf{Case (i) ($\sigma=(a,b,c)$):}} Consider the permutation $(123)$. Then \[C_{A_7}((123))=\langle (123)\rangle\times\langle (456),(45)(67)\rangle.\] Clearly, $\{e,(132),(45)(67),(46)(57),(47)(56),(123)(45)(67),(123)(46)(57),(123)(47)(56),\\ (132)(45)(67),(132)(46)(57),(132)(47)(56)\}=\mathcal{N}_{\mathcal{P}_e(A_7)}((123))\subseteq \mathcal{N}_{\Delta_D(A_7)}((123))$. Then, any permutation in $C_{A_7}((123))\setminus\mathcal{N}_{\mathcal{P}_e(A_7)}[(123)]$ is one of the form $\lambda, (123)\lambda$ and $(132)\lambda$, where $\lambda$ is a $3$-cycle disjoint from $(123)$. Now from Lemma \ref{disjoint_commutators}, we have $[\tilde{(123)},\tilde{\lambda}]=[\tilde{(123)},\tilde{(123)}\tilde{\lambda}]=[\tilde{(123)},\tilde{(132)}\tilde{\lambda}]$, and from Lemma \ref{3-cycle A_7}
we have $[\tilde{(123)},\tilde{\lambda}]\neq e$. Hence, $\mathcal{N}_{\Delta_D(A_7)}((123))=\mathcal{N}_{\mathcal{P}_e(A_7)}((123))$. Similarly, it can be shown that for any arbitrary permutation of the form $(a,b,c)\in A_7$, we have $\mathcal{N}_{\Delta_D(A_7)}((a,b,c))=\mathcal{N}_{\mathcal{P}_e(A_7)}((a,b,c))$.

\vspace{.3cm}
\noindent\underline{\bf{Case (ii) ($\sigma=(a,b)(c,d)$):}} Consider the permutation $(12)(34)$. Then
\begin{center}
$C_{A_7}((12)(34))=$\\
$(\langle (12)(34),(13)(24)\rangle\times\langle (567)\rangle)\cup(\{(12),(34),(1324),(1423)\}\times\{(56),(57),(67)\}).$   
\end{center}
\vspace{.15cm}
Clearly, $\langle(12)(34)(567)\rangle\cup(\{(1324),(1423)\}\times\{(56),(57),(67)\})=\mathcal{N}_{\mathcal{P}_e(A_7)}[(12)(34)]\subseteq \\ \mathcal{N}_{\Delta_D(A_7)}[(12)(34)]$. Hence, $C_{A_7}((12)(34))\setminus\mathcal{N}_{\mathcal{P}_e(A_7)}[(12)(34)]=(\{(13)(24),(14)(23)\}\times\langle(567)\rangle)\cup(\{(12),(34)\}\times\{(56),(57),(67)\})$. Now, since $(567),(576)\in\mathcal{N}_{\Delta_D(A_7)}[(12)(34)]$, we have $[\tilde{(12)(34)},\tilde{(567)}]=[\tilde{(12)(34)},\tilde{(576)}]=e$. Then from Lemma \ref{disjoint_commutators}, we can say that for any $v\in (\{(13)(24),(14)(23)\}\times\langle(567)\rangle)\cup(\{(12),(34)\}\times\{(56),(57),(67)\})$, $[\tilde{(12)(34)},\tilde{v}]=[\tilde{(12)(34)},\tilde{\xi}]$, for some $\xi\in\{(13)(24),(14)(23)\}\cup(\{(12),(34)\}\times\{(56),(57),(67)\})$. Now, from Lemma \ref{2-cycle A_7} we have $[\tilde{(12)(34)},\tilde{\xi}]\neq e$. Hence, $\mathcal{N}_{\Delta_D(A_7)}((12)(34))=\mathcal{N}_{\mathcal{P}_e(A_7)}((12)(34))$. Similarly, it follows that for any arbitrary permutation of the form $(a,b)(c,d)\in A_7$, we have $\mathcal{N}_{\Delta_D(A_7)}((a,b)(c,d))=\mathcal{N}_{\mathcal{P}_e(A_7)}((a,b)(c,d))$.  

\vspace{.3cm}
\noindent\underline{\bf{Case (iii) ($\sigma=(a,b,c)(d,e,f)$):}} Consider the permutation $(123)(456)$. Then \[C_{A_7}((123)(456))=\langle(123)\rangle\times\langle(456)\rangle.\] Here, $\langle(123)(456)\rangle=\mathcal{N}_{\mathcal{P}_e(A_7)}[(123)(456)]\subseteq\mathcal{N}_{\Delta_D(A_7)}[(123)(456)]$. For the rest of the permutations in $C_{A_7}((123)(456))$, using Lemma \ref{disjoint_commutators} and \ref{3-cycle A_7} we can say the following.
\begin{itemize}
    \item $[\tilde{(123)}\tilde{(456)},\tilde{(123)}]=[\tilde{(456)},\tilde{(123)}]\neq e.$
    \item $[\tilde{(123)}\tilde{(456)},\tilde{(132)}]=[\tilde{(456)},\tilde{(132)}]\neq e.$
    \item $[\tilde{(123)}\tilde{(456)},\tilde{(456)}]=[\tilde{(123)},\tilde{(456)}]\neq e.$
    \item $[\tilde{(123)}\tilde{(456)},\tilde{(465)}]=[\tilde{(123)},\tilde{(465)}]\neq e.$
    \item $[\tilde{(123)}\tilde{(456)},\tilde{(123)}\tilde{(465)}]=[\tilde{(123)},\tilde{(465)}][\tilde{(456)},\tilde{(123)}]=[\tilde{(123)},\tilde{(456)}]\neq e.$
    \item $[\tilde{(123)}\tilde{(456)},\tilde{(132)}\tilde{(456)}]=[\tilde{(123)},\tilde{(456)}][\tilde{(456)},\tilde{(132)}]=[\tilde{(456)},\tilde{(123)}]\neq e.$
\end{itemize}
Hence, we have $\mathcal{N}_{\Delta_D(A_7)}((123)(456))=\mathcal{N}_{\mathcal{P}_e(A_7)}((123)(456))$.

Similarly, it can be shown that for any arbitrary permutation of the form $(a,b,c)(d,e,f)\in A_7$, we have $\mathcal{N}_{\Delta_D(A_7)}((a,b,c)(d,e,f))=\mathcal{N}_{\mathcal{P}_e(A_7)}((a,b,c)(d,e,f))$.

\vspace{.3cm}
\noindent\underline{\bf{Case (iv) ($\sigma=(a,b,c)(d,e)(f,g)$):}} Consider the permutation $(123)(45)(67)$. Then \[C_{A_7}((123)(45)(67))=\langle(123)\rangle\times\langle(45)(67),(46)(57)\rangle.\] Here, $\langle(123)(45)(67)\rangle=\mathcal{N}_{\mathcal{P}_e(A_7)}[(123)(45)(67)]\subseteq\mathcal{N}_{\Delta_D(A_7)}[(123)(45)(67)]$. Now, for any $u\in \langle(123),(132)\rangle$ and $v\in \langle(46)(57),(47)(56)\rangle$, we have $v\in\mathcal{N}_{\mathcal{P}_e(A_7)}[u]\subseteq\mathcal{N}_{\Delta_D(A_7)}[u]$. Hence, $[\tilde{u},\tilde{v}]=e$. Now, from Lemma \ref{disjoint_commutators} we can see that 
\begin{itemize}
    \item $[\tilde{(123)}\tilde{(45)(67)},\tilde{(46)(57)}]=[\tilde{(45)(67)},\tilde{(46)(57)}].$
    \item $[\tilde{(123)}\tilde{(45)(67)},\tilde{(123)}\tilde{(46)(57)}]=[\tilde{(45)(67)},\tilde{(46)(57)}]$.
    \item $[\tilde{(123)}\tilde{(45)(67)},\tilde{(132)}\tilde{(46)(57)}]=[\tilde{(45)(67)},\tilde{(46)(57)}].$
    \item $[\tilde{(123)}\tilde{(45)(67)},\tilde{(47)(56)}]=[\tilde{(45)(67)},\tilde{(47)(56)}]$.
    \item $[\tilde{(123)}\tilde{(45)(67)},\tilde{(123)}\tilde{(47)(56)}]=[\tilde{(45)(67)},\tilde{(47)(56)}]$.
    \item $[\tilde{(123)}\tilde{(45)(67)},\tilde{(132)}\tilde{(47)(56)}]=[\tilde{(45)(67)},\tilde{(47)(56)}]$.
\end{itemize}
In addition, Lemma \ref{2-cycle A_7} yields $[\tilde{(45)(67)},\tilde{(46)(57)}]\neq e$ and $[\tilde{(45)(67)},\tilde{(47)(56)}]\neq e$. Hence, we have $\mathcal{N}_{\Delta_D(A_7)}((123)(45)(67))=\mathcal{N}_{\mathcal{P}_e(A_7)}((123)(45)(67))$. Now, the rest follows via similar arguments.

\vspace{.3cm}
\noindent\underline{\bf{Case (v) ($\sigma=(a,b,c,d)(e,f)$):}} Consider the permutation $(1234)(56)$. Then \[C_{A_7}((1234)(56))=\langle(1234)(56)\rangle.\] Here, $\mathcal{N}_{\mathcal{P}_e(A_7)}[(1234)(56)]=C_{A_7}((1234)(56))$. Hence, it follows that $\mathcal{N}_{\mathcal{P}_e(A_7)}((1234)(56))=\mathcal{N}_{\Delta_D(A_7)}((1234)(56))$. Now, the rest follows similarly.

\vspace{.3cm}
\noindent\underline{\bf{Case (vi) ($\sigma=(a,b,c,d,e)$):}} Consider the permutation $(12345)$. Then \[C_{A_7}((12345))=\langle(12345)\rangle.\] Here, $\mathcal{N}_{\mathcal{P}_e(A_7)}[(12345)]=C_{A_7}((12345))$. Hence, it is not difficult to see that $\mathcal{N}_{\mathcal{P}_e(A_7)}((12345))=\mathcal{N}_{\Delta_D(A_7)}((12345))$. The rest follows similarly.

\vspace{.3cm}
\noindent\underline{\bf{Case (vii) ($\sigma=(a,b,c,d,e,f,g)$):}} Consider the permutation $(1234567)$. Then \[C_{A_7}((1234567))=\langle(1234567)\rangle.\] Here, $\mathcal{N}_{\mathcal{P}_e(A_7)}[(1234567)]=C_{A_7}((1234567))$. Hence, it follows that $\mathcal{N}_{\mathcal{P}_e(A_7)}((1234567))=\mathcal{N}_{\Delta_D(A_7)}((1234567))$. Again, the rest follows similarly as above.

Hence, we have $\Delta_D(A_7)=\mathcal{P}_e(A_7)$. Now, from Proposition \ref{induced_S_n and A_n}(iii) and the fact that the induced subgraph of $\mathcal{P}_e(G)\vert_S$, for any subgroup $S$ of $G$ coincides with $\mathcal{P}_e(S)$, the equality can be deduced for $A_6$. This concludes the proof. 
\end{proof}

\begin{theorem}\label{Containment of Delta_D for A_n}
For a positive integer $n\geq 4$, we have the following:
\begin{enumerate}[\rm (i)]
\item $E(\mathcal{P}_e(A_n))=E(\Delta_D(A_n))\subsetneq E(\Delta(A_n))$ when $n\leq 7$;
\item $E(\mathcal{P}_e(A_n))\subsetneq E(\Delta_D(A_n))\subsetneq E(\Delta(A_n))$ when $n\geq 8$.
\end{enumerate}
\end{theorem}

\begin{proof}
For $n\geq 4$, from Proposition \ref{induced_over_alternating} it follows that for $n\notin \{6,7\}$, $(12)(34)\sim (13)(24)$ in $\Delta_D(A_n)$ if and only if $(12)(34)\sim (13)(24)$ in $\Delta_D(S_n)$. Now, from the proof of Proposition \ref{deep commuting of S_4,5}, it follows that $(12)(34)\nsim (13)(24)$ in $\Delta_D(S_n)$ for all $n\geq 4$, hence $(12)(34)\nsim (13)(24)$ in $\Delta_D(A_n)$ for all $n\geq 4$ with $n\notin \{6,7\}$. Also for $n\in\{6,7\}$, it follows from the proof of Proposition \ref{deep commuting of A_6,7} that $(12)(34)\nsim (13)(24)$ in $\Delta_D(A_n)$. But, since $[(12)(34),(13)(24)]=e$, we have $(12)(34)\sim (13)(24)$ in $\Delta(A_n)$. Hence $E(\Delta_D(A_n))\subsetneq E(\Delta(A_n))$, for all $n\geq 4$.   

Recall that the induced subgraph $\mathcal{P}_e(G)\vert_S$, for any subgroup $S$ of $G$ coincides with $\mathcal{P}_e(S)$. Then for $n\in\{4,5\}$, it follows from Proposition \ref{induced_over_alternating} and \ref{deep commuting of S_4,5} that $\mathcal{P}_e(A_n)=\Delta_D(A_n)$. Also, it follows from Proposition \ref{deep commuting of A_6,7} that $\mathcal{P}_e(A_n)=\Delta_D(A_n)$, for $n\in\{6,7\}$. Now, for $n\geq 8$, we have $(123),(456)\in A_n$, and from Corollary \ref{disjoint_permutations_in_An} we have $(123)\sim(456)$ in $\Delta_D(A_n)$. Now, since $\langle(123),(456)\rangle\cong C_3\times C_3$, we have $(123)\nsim(456)$ in $\mathcal{P}_e(A_n)$. Hence, for $n\geq 8$, $E(\mathcal{P}_e(A_n))\subsetneq E(\Delta_D(A_n))$. This completes the proof.  
\end{proof}
Our next result determines positive integers $n$ for which $\Delta_D(S_n)$ and $\Delta_D(A_n)$ are perfect.
\begin{theorem}\label{permutation_perfect}
For $n\geq 4$, we have the following:
\begin{enumerate}[\rm (i)]
    \item $\Delta_D(S_n)$ is perfect if and only if $n\leq 5$.
    \item $\Delta_D(A_n)$ is perfect if and only if $n\leq 8$.
\end{enumerate}
\end{theorem}
\begin{proof}
Let $n\geq 4$ be a positive integer. 
\begin{enumerate}[\rm (i)]
    \item For $n\geq 6$, using Lemma \ref{disjoint_permutations_in_Sn} it can be seen that the subgraph induced in $\Delta_D(S_n)$ by the set $\{(123),(456),(12),(12)(34)(56),(56)\}$ is a cycle of length $5$, which is not perfect. Hence, for any $n\geq 6$, $\Delta_D(S_n)$ is not perfect. Also, from Proposition \ref{deep commuting of S_4,5} we have $\Delta_D(S_5)=\mathcal{P}_e(S_5)$. Now, it is easy to see that 
    \begin{itemize}
        \item $\mathcal{N}_{\mathcal{P}_e(S_5)}[(a,b,c)(d,e)]=\mathcal{N}_{\mathcal{P}_e(S_5)}[(a,b,c)]=\mathcal{N}_{\mathcal{P}_e(S_5)}[(d,e)]=\langle(a,b,c)(d,e)\rangle$ in $\mathcal{P}_e(S_5)$, where $(a,b,c)$ and $(d,e)$ are disjoint.
        \item $\mathcal{N}_{\mathcal{P}_e(S_5)}[(a,b,c,d)]=\mathcal{N}_{\mathcal{P}_e(S_5)}[(a,c)(b,d)]=\langle(a,b,c,d)\rangle$ in $\mathcal{P}_e(S_5)$.
        \item $\mathcal{N}_{\mathcal{P}_e(S_5)}[(a,b,c,d,e)]=\langle(a,b,c,d,e)\rangle$ in $\mathcal{P}_e(S_5)$.
    \end{itemize}  
    Hence, $\Delta_D(S_5)$ is a generalized join of a star graph by complete graphs. So by Theorem \ref{perfectness of generalized join}, $\Delta_D(S_5)$ is perfect. Also, Proposition \ref{induced_S_n and A_n} implies that $\Delta_D(S_4)$ is an induced subgraph of $\Delta_D(S_5)$. Hence $\Delta_D(S_4)$ is perfect.
    
    \item For $n\geq 8$, it was shown in the beginning of Section \ref{perfectness} that $\Delta_D(A_n)$ is not perfect. Also, from Theorem \ref{Simple group perfect} we have $\Delta_D(A_n)$ is perfect for $n\in \{5,6,7\}$. Moreover, Proposition \ref{induced_S_n and A_n} shows that $\Delta_D(A_4)$ is an induced subgraph of $\Delta_D(S_4)$, hence $\Delta_D(A_4)$ is perfect. This concludes the proof.
\end{enumerate}
\end{proof}
For $n\geq 4$, it is well known that $Z(S_n)=Z(A_n)=\{e\}$. Since, for any finite group $G$, ${_d}\Delta_D(G)\subseteq {_d}\Delta(G)$, we can say the following: 
\begin{proposition}\label{Sn_An_dominant}
${_d}\Delta_D(S_n)={_d}\Delta_D(A_n)=\{e\}.$    
\end{proposition}

From Theorem \ref{Containment of Delta_D for S_n} and Theorem \ref{Containment of Delta_D for A_n}, it can be seen that $\Delta_D(S_n)$ coincides with $\mathcal{P}_e(S_n)$ for $n\leq 5$ and $\Delta_D(A_n)$ coincides with $\mathcal{P}_e(A_n)$ for $n\leq 7$. In addition to that, the connectivity of $\mathcal{P}_e(S_n)^*$ and $\mathcal{P}_e(A_n)^*$ were well explored in \cite{bera}. So, our next results consider the connectedness of the reduced graph of $\Delta_D(S_n)$, for $n\geq 6$ and the reduced graph of $\Delta_D(A_n)$, for $n\geq 8$. In particular, we have the following.

\begin{theorem}\label{Reduced_Sn_connected}
For a positive integer $n\geq 6$, $\Delta_D(S_n)^*$ is connected if and only if neither $n$ nor $n-1$ is a prime.
\end{theorem}
\begin{proof}
Let one of $n$ or $n-1$ be a prime $p$. Let $\sigma$ be a $p$-cycle in $S_n$. Then it is easy to see that $\mathcal{N}_{\mathcal{P}_e(S_n)}[\sigma]=\mathcal{N}_{\Delta_D(S_n)}[\sigma]=C_{S_n}(\sigma)=\langle\sigma\rangle$. Also for any $i\in[p-1]$, $\sigma^i$ is again a $p$-cycle, hence $\mathcal{N}_{\Delta_D(S_n)}[\sigma^i]=\langle\sigma\rangle$. So, we have $\mathcal{N}_{\Delta_D(S_n)}[\langle\sigma\rangle\setminus \{e\}]=\underset{i\in[p-1]}{\cup}\mathcal{N}_{\Delta_D(S_n)}[\sigma^i]=\langle\sigma\rangle$. Hence, $\langle\sigma\rangle\setminus \{e\}$ and $\Delta_D(S_n)^*\setminus\{\langle\sigma\rangle\setminus \{e\}\}$ forms a separation of $\Delta_D(S_n)^*$. So $\Delta_D(S_n)^*$ is disconnected. 

Conversely, let both $n$ and $n-1$ be composite. We will show that $\Delta_D(S_n)^*$ is connected in three steps. First, we will show that any vertex of odd order of $\Delta_D(S_n)^*$ is either itself a cyclic permutation of length $q$ or connected to a cyclic permutation of length $q$, for some odd prime $q$. Also any vertex of even order of $\Delta_D(S_n)^*$ is either itself a permutation of the form $(a,b)(c,d)$ where $a,b,c,d\in[n]$ are all distinct, or connected to a permutation of the above form. Next we show that any cyclic permutation of odd prime length or any permutation of the form $(a,b)(c,d)$ is adjacent to some transposition in $\Delta_D(S_n)^*$. We conclude the proof by showing that any two distinct transpositions in $\Delta_D(S_n)^*$ are connected. 

\noindent\underline{\bf{Step 1:}} Let $\tau\in S_n$ be a non-trivial permutation. If $\tau$ is of odd order, then either $\tau$ is of an odd prime order, or there exist some positive integer $k$ such that $\tau^k$ is of an odd prime order. We denote that odd prime by $q$. Now, any element of order $q$ in $S_n$ is either a single $q$-cycle or a product of disjoint $q$-cycles. If $\tau$ or $\tau^k$ is a single $q$-cycle, then we are done. Otherwise, we assume it to be of the form $\delta_1\delta_2\dots\delta_m$, where $m\geq 2$ and $\delta_i$'s are disjoint $q$-cycles. Now, choose the $q$-cycle $\delta_1$. Then from Lemma \ref{disjoint_commutators} and \ref{disjoint_permutations_in_Sn}, we have  $[\tilde{\delta_1}\tilde{\delta_2}\dots\tilde{\delta_m},\tilde{\delta_1}]=[\tilde{\delta_1},\tilde{\delta_1}][\tilde{\delta_2}\dots\tilde{\delta_m},\tilde{\delta_1}]=e$. Then $\tau\sim \delta_1$ or $\tau\sim\tau^k\sim\delta_1$ forms a path from $\tau$ to a $q$-cycle in $\Delta_D(S_n)^*$.

Again, if $\tau$ is of even order, then either $\tau$ is of order $2$, or there exists some positive integer $l$, such that $\tau^l$ is of order $2$. Now, any element of order $2$ in $S_n$ is either a single transposition or a product of disjoint transpositions. Let $\tau$ or $\tau^l$ be of the form $\epsilon_1\epsilon_2\dots\epsilon_t$, where $\epsilon_i$'s are disjoint transpositions. If $t=2$, then we are done. If $t=1$, then there exists a permutation $(a,b)(c,d)\in S_n$ such that $a,b,c,d\in[n]$ are all distinct and $(a,b)(c,d)$ is disjoint from $\epsilon_1$. Then from Lemma \ref{disjoint_permutations_in_Sn} we have $[\tilde{\epsilon_1},\tilde{(a,b)(c,d)}]=e$. In addition, when $t\geq 3$, Lemma \ref{disjoint_commutators} and \ref{disjoint_permutations_in_Sn} implies $[\tilde{\epsilon_1}\tilde{\epsilon_2}\dots\tilde{\epsilon_t},\tilde{\epsilon_1}\tilde{\epsilon_2}]=[\tilde{\epsilon_1}\tilde{\epsilon_2},\tilde{\epsilon_1}\tilde{\epsilon_2}][\tilde{\epsilon_3}\dots\tilde{\epsilon_t},\tilde{\epsilon_1}\tilde{\epsilon_2}]=e$. So, in all possible scenarios, we have either $\tau$ is of the form $(a,b)(c,d)$, or there exists a path from $\tau$ to a permutation of the form $(a,b)(c,d)$ in $\Delta_D(S_n)^*$. 

\noindent\underline{\bf{Step 2:}} Now consider a permutation $\lambda\in S_n$ such that either $\lambda$ is a $q$-cycle for some odd prime $q$, or $\lambda$ is a permutation of the form $(a,b)(c,d)$ where $a,b,c,d\in[n]$ are all distinct. Clearly $\lambda$ is an even permutation. Now, since $n\geq 6$ and neither $n$ nor $n-1$ is a prime, there exist a transposition $\epsilon\in S_n$ which is disjoint from $\lambda$. Then from Lemma \ref{disjoint_permutations_in_Sn} we have $\lambda\sim\epsilon$ in $\Delta_D(S_n)^*$. 

\noindent\underline{\bf{Step 3:}} Now, let $(a_1,b_1)$ and $(a_2,b_2)$ be two distinct transpositions in $S_n$. If $(a_1,b_1)$ and $(a_2,b_2)$ are disjoint, then there exist a permutation $(a_3,b_3)\in S_n$ such that $(a_3,b_3)$ is disjoint from both $(a_1,b_1)$ and $(a_2,b_2)$, since $n\geq 6$. Then from Lemma \ref{disjoint_commutators} and \ref{disjoint_permutations_in_Sn} we have, $(a_1,b_1)\sim (a_1,b_1)(a_2,b_2)(a_3,b_3)\sim(a_2,b_2)$ in $\Delta_D(S_n)^*$. Also, if $(a_1,b_1)$ and $(a_2,b_2)$ are not disjoint, then there exist a permutation $(c_1,c_2,c_3)\in S_n$ such that $(c_1,c_2,c_3)$ is disjoint from both $(a_1,b_1)$ and $(a_2,b_2)$. Then from Lemma \ref{disjoint_permutations_in_Sn} we have, $(a_1,b_1)\sim (c_1,c_2,c_3)\sim(a_2,b_2)$ in $\Delta_D(S_n)^*$.

Hence, for any two non trivial permutations $\tau_1, \tau_2\in S_n$, we have a path $\tau_1\sim\dots\sim\lambda_1\sim\xi_1\sim\dots\sim\xi_2\sim\lambda_2\sim\dots\sim\tau_2$ in $\Delta_D(S_n)^*$, where $\lambda_1,\lambda_2$ are either cycles of odd prime length or permutations of the form $(a,b)(c,d)$ with $(a,b)$ and $(c,d)$ being disjoint, and $\xi_1,\xi_2$ are transpositions. This concludes the proof.
\end{proof}

In the above proof, it is worth noticing that any two permutations having orders divisible by some prime less that $n-1$ are connected by a path in $\Delta_D(S_n)^*$. With this observation in hand, our next result counts the number of components of $\Delta_D(S_n)^*$, when $\Delta_D(S_n)^*$ is disconnected.

\begin{corollary}\label{Reduced_Sn_components}
 For a positive integer $n\geq 6$, let $\Delta_D(S_n)^*$ be disconnected. Then we have the following:
 \begin{enumerate}[\rm (i)]
     \item If $n$ is prime, then $\Delta_D(S_n)^*$ has $(n-2)!+1$ components;
     \item If $n-1$ is prime, then $\Delta_D(S_n)^*$ has $n(n-3)!+1$ components.
 \end{enumerate}
\end{corollary}

\begin{proof}
 $n\geq 6$ is a positive integer, and $\Delta_D(S_n)^*$ is disconnected. Then from Theorem \ref{Reduced_Sn_connected}, we have either $n$ or $n-1$ is prime. Let that prime be $p$. Then, it is easy to see that for any permutation $\sigma\in S_n$, either $\sigma$ is a $p$-cycle, or $o(\sigma)$ is divisible by some prime less than $p$. Using similar arguments as in the proof of Theorem \ref{Reduced_Sn_connected}, it can be concluded that all elements other that $p$-cycles in $S_n$ form a single connected component of $\Delta_D(S_n)^*$. Now,  
 \begin{enumerate}[\rm (i)]
     \item If $p=n$, then it was shown in the proof of Theorem \ref{Reduced_Sn_connected} that all the generators of subgroup generated by a $p$-cycles, together form a maximal connected component of $\Delta_D(S_n)^*$. Since there are exactly $(p-1)!$ distinct $p$-cycles in $S_p$ and each subgroup generated by $p$-cycles have exactly $p-1$ generators, we can conclude that the subgraph induced in $\Delta_D(S_n)^*$ by the set of all $p$-cycles have $(p-2)!$ components. So, $\Delta_D(S_n)^*$ have total $(n-2)!+1$ components.
     \item $p=n-1$, then there are exactly $(p+1)(p-1)!$ distinct $p$-cycles in $S_p$ and each subgroup generated by $p$-cycles have exactly $p-1$ generators. Using similar arguments as in the first case, we can conclude that the subgraph induced in $\Delta_D(S_n)^*$ by the set of all $p$-cycles have $(p+1)(p-2)!$ components. So, $\Delta_D(S_n)^*$ have total $n(n-3)!+1$ components.
 \end{enumerate}
\end{proof}

\begin{theorem}\label{Reduced_An_connected}
For a positive integer $n\geq 8$, $\Delta_D(A_n)^*$ is connected if and only if none of $n,n-1$ and $n-2$ is prime.  
\end{theorem}
\begin{proof}
This proof is similar to the proof of Theorem \ref{Reduced_Sn_connected}. Let at least one of $n, n-1$ and $n-2$ be a prime. Denote that prime by $p$. Let $\sigma$ be a $p$-cycle in $A_n$. Then it is easy to see that $\mathcal{N}_{\mathcal{P}_e(A_n)}[\sigma]=\mathcal{N}_{\Delta_D(A_n)}[\sigma]=C_{A_n}(\sigma)=\langle\sigma\rangle$. Also, for any $i\in[p-1]$, $\sigma^i$ is again a $p$-cycle, hence $\mathcal{N}_{\Delta_D(A_n)}[\sigma^i]=\langle\sigma\rangle$. So, we have $\mathcal{N}_{\Delta_D(A_n)}[\langle\sigma\rangle\setminus \{e\}]=\underset{i\in[p-1]}{\cup}\mathcal{N}_{\Delta_D(A_n)}[\sigma^i]=\langle\sigma\rangle$. Hence, $\langle\sigma\rangle\setminus \{e\}$ and $\Delta_D(A_n)^*\setminus\{\langle\sigma\rangle\setminus \{e\}\}$ forms a separation of $\Delta_D(A_n)^*$. So $\Delta_D(A_n)^*$ is disconnected. 

Conversely, let each of $n, n-1$ and $n-2$ be composite.

\noindent\underline{\bf{Step 1:}} Let $\tau\in A_n$ be a non-trivial permutation. Then using similar arguments as in Step 1 of the proof of Theorem \ref{Reduced_Sn_connected}, we can show that $\tau$ is connected to a $q$-cycle for some odd prime $q$, or $\tau$ is connected to a permutation of the form $(a,b)(c,d)\in A_n$.

\noindent\underline{\bf{Step 2:}} Next, consider a permutation $\lambda\in A_n$ such that $\lambda$ is a $q$-cycle for some odd prime $q$, or $\lambda$ is a permutation of the form $(a,b)(c,d)$ where $a,b,c,d\in[n]$ are all distinct. Now, since $n\geq 8$ and none of $n,n-1$ and $n-2$ is a prime, there exists a $3$-cycle $\gamma\in A_n$ which is disjoint from $\lambda$. Then from Corollary \ref{disjoint_permutations_in_An} we have $\lambda\sim\gamma$ in $\Delta_D(A_n)^*$. 

\noindent\underline{\bf{Step 3:}} Now, let $(a_1,b_1,c_1)$ and $(a_2,b_2,c_2)$ be two distinct $3$-cycles in $A_n$. If $(a_1,b_1,c_1)$ and $(a_2,b_2,c_2)$ are disjoint, then from Corollary \ref{disjoint_permutations_in_An} we have $(a_1,b_1,c_1)\sim (a_2,b_2,c_2)$ in $\Delta_D(A_n)^*$. If $(a_1,b_1,c_1)$ and $(a_2,b_2,c_2)$ are not disjoint, then there exist a permutation $(a_3,b_3,c_3)\in A_n$ such that $(a_3,b_3,c_3)$ is disjoint with both $(a_1,b_1,c_1)$ and $(a_2,b_2,c_2)$, since $n\geq 8$. So Corollary \ref{disjoint_permutations_in_An} implies that there is a path $(a_1,b_1,c_1)\sim(a_3,b_3,c_3)\sim(a_2,b_2,c_2)$ in $\Delta_D(A_n)^*$.

Hence, for any two non trivial permutations $\tau_1, \tau_2\in A_n$, we have a path $\tau_1\sim\dots\sim\lambda_1\sim\gamma_1\sim\dots\sim\gamma_2\sim\lambda_2\sim\dots\sim\tau_2$ in $\Delta_D(A_n)^*$, where $\lambda_1,\lambda_2$ are either cycles of odd prime length or permutations of the form $(a,b)(c,d)$, and $\gamma_1,\gamma_2$ are $3$-cycles. This concludes the proof.   
\end{proof}

Using similar arguments as before, when $\Delta_D(A_n)^*$ is disconnected, the number of connected components of $\Delta_D(A_n)^*$ can also be computed.

\begin{corollary}\label{Reduced_An_components}
 For a positive integer $n\geq 8$, let $\Delta_D(A_n)^*$ be disconnected. Then we have the following:
 \begin{enumerate}[\rm (i)]
     \item If only $n$ is prime, then $\Delta_D(A_n)^*$ has $(n-2)!+1$ components;
     \item If only $n-1$ is prime, then $\Delta_D(A_n)^*$ has $n(n-3)!+1$ components;
     \item If only $n-2$ is prime, then $\Delta_D(A_n)^*$ has $\frac{n(n-1)(n-4)!}{2}+1$ components;
     \item If both $n$ and $n-2$ are primes, then $\Delta_D(A_n)^*$ has $(n-2)!+\frac{n(n-1)(n-4)!}{2}+1$ components.
 \end{enumerate}
\end{corollary}

\section{Conclusion and further directions}
In this article, we explored two aspects of the deep commuting graph of finite groups. One part of our study revolves around classifying finite groups for which the deep commuting graph coincides with the enhanced power graph or the commuting graph. The other aspect is studying several combinatorial properties of the deep commuting graph, e.g.~completeness, universality, Eulerian circuits, perfectness, connectivity, diameter etc. Here are some open ends and further directions, that have arisen from our study.
\begin{itemize}
    \item For finite nilpotent groups of odd order, we found a necessary condition for which the deep commuting graph coincides with the enhanced power graph. Yet, a complete classification of finite nilpotent groups, or in general any finite group for which the deep commuting graph coincides with the enhanced power graph is still open. A similar classification for the commuting graph is also not fully covered. 

    \item A complete classification of finite groups with perfect deep commuting graph is yet to be done.

    \item In this article, we provided upper bounds for the diameter of the reduced graphs (if connected) of the deep commuting graph of finite abelian groups and Heisenberg groups. An exact enumeration of these diameters is yet to be fully explored. Also, enumeration of the exact diameter of the reduced graph (if connected) of the deep commuting graph of $S_n$ and $A_n$ is still open.

    \item We found several finite groups, for which the reduced graphs of the deep commuting graph are connected. An exact enumeration of the vertex connectivity of the deep commuting graphs of these groups can be studied further. In particular, an exact enumeration of the vertex connectivity of the deep commuting graph of finite abelian groups is an interesting problem to explore. 
\end{itemize}

\bibliographystyle{plain}

\vskip .5cm

\noindent{\bf Addresses}: Sumana Hatui, Sanjay Mukherjee, Kamal Lochan Patra

\begin{enumerate}
\item[1)] School of Mathematical Sciences\\
National Institute of Science Education and Research (NISER), Bhubaneswar\\
P.O.- Jatni, District- Khurda, Odisha--752050, India

\item[2)] Homi Bhabha National Institute (HBNI)\\
Training School Complex, Anushakti Nagar, Mumbai--400094, India
\end{enumerate}

\noindent {\bf E-mails}: sumanahatui@niser.ac.in, sanjay.mukherjee@niser.ac.in, klpatra@niser.ac.in
\end{document}